\documentclass[12pt]{amsart} 
\usepackage{geometry}             

\usepackage{float}
\usepackage{bm}
\usepackage{fullpage,xcolor}
\usepackage[mathscr]{euscript}
\usepackage[all]{xy}
\usepackage{epsfig}
\usepackage[T1]{fontenc}
\usepackage{amsfonts}

\usepackage{graphicx}
\usepackage{amssymb}
\usepackage{amsmath}
\usepackage{amsthm}
\usepackage{mathrsfs}
\usepackage{epstopdf}
\usepackage{url}
\usepackage{hyperref}
\usepackage[msc-links,alphabetic]{amsrefs}
\usepackage{tikz}

\usepackage{color}

\newcommand{\be}{\beta}
\newcommand{\ga}{\gamma}

\newcommand{\ep}{\varepsilon}

\newcommand{\si}{\sigma}

\newcommand{\Ga}{\Gamma}

\newcommand{\Si}{\Sigma}
\newcommand{\ZZ}{{\mathbb Z}}

\newcommand{\RR}{{\mathbb R}}

\newcommand{\sS}{\mathfrak S}
\newcommand{\sL}{\mathscr L}

\newcommand{\lto}{\longrightarrow}
\newcommand{\lk}{\operatorname{\ell{\it k}}}

\newcommand{\genus}{\operatorname{genus}}

\newcommand{\sm}{\smallsetminus}
\newcommand{\co}{\colon}

\newcommand{\wt}[1]{{\widetilde{#1}}}
\newcommand{\spn}{\operatorname{span}}

\newcommand{\lb}{\left\langle}
\newcommand{\rb}{\right\rangle}

\newcommand{\KP}[1]{%
  \begin{tikzpicture}[baseline=-\dimexpr\fontdimen22\textfont2\relax]
  #1
  \end{tikzpicture}%
}
\newcommand{\KPC}{%
  \KP{\filldraw[color=black, fill=none, very thick] circle (0.18);}%
}
\newcommand{\KPX}{%
  \KP{
    \draw[color=black,very thick] (-0.3,-0.3) -- (0.3,0.3);
    \draw[color=black,very thick] (-0.3,0.3) -- (-0.05,0.05);
    \draw[color=black,very thick] (0.05,-0.05) -- (0.3,-0.3);
  }%
}
\newcommand{\KPB}{%
  \KP{%
    \draw[color=black,very thick] (-0.3,0.3) .. controls (0,-0.02) .. (0.3,0.3);
    \draw[color=black,very thick] (-0.3,-0.3) .. controls (0,0.02) .. (0.3,-0.3);
  }%
}
\newcommand{\KPA}{%
  \KP{%
    \draw[color=black,very thick] (-0.3,-0.3) .. controls (0.02,0) .. (-0.3,0.3);
    \draw[color=black,very thick] (0.3,-0.3) .. controls (-0.02,0) .. (0.3,0.3);
  }%
}
\newcommand{\KPD}{%
  \KP{%
    \draw[color=black,very thick] (-0.6,-0.30) .. controls (-0.5,-0.12) .. (-0.35,-0.00);
    \draw[color=black,very thick] (0.6,-0.30) .. controls (0.5,-0.12) .. (0.35,-0.00);
    \draw[color=black,very thick] (-0.24,0.08) .. controls (0.0, 0.20) .. (0.24, 0.08);
    \draw[color=black,very thick] (0.6,0.30) .. controls (0.0,-0.24) .. (-0.6,0.30);
  }%
}
\newcommand{\KPE}{%
  \KP{%
    \draw[color=black,very thick] (-0.45,-0.30) .. controls (0.0,-0.00) .. (0.45,-0.30);
    \draw[color=black,very thick] (-0.45,0.30) .. controls (0.0,0.00) .. (0.45,0.30);
  }%
}
\newcommand{\KPXY}{%
  \KP{
    \draw[color=black,very thick] (-0.4,0.4) -- (0.4,-0.4);
    \draw[color=black,very thick] (-0.4,-0.4) -- (-0.05,-0.05);
    \draw[color=black,very thick] (0.05,0.05) -- (0.4,0.4);
    \draw[color=black,very thick] (-0.5,0.05) .. controls (-0.4,0.12) .. (-0.27,0.20);
    \draw[color=black,very thick] (0.5,0.05) .. controls (0.4,0.12) .. (0.27,0.20);
    \draw[color=black,very thick] (-0.14,0.24) .. controls (0.0,0.26) .. (0.14,0.24);
  }%
}
\newcommand{\KPXZ}{%
  \KP{
    \draw[color=black,very thick] (-0.4,0.4) -- (0.4,-0.4);
    \draw[color=black,very thick] (-0.4,-0.4) -- (-0.05,-0.05);
    \draw[color=black,very thick] (0.05,0.05) -- (0.4,0.4);
    \draw[color=black,very thick] (-0.5,-0.05) .. controls (-0.4,-0.12) .. (-0.27,-0.20);
    \draw[color=black,very thick] (0.5,-0.05) .. controls (0.4,-0.12) .. (0.27,-0.20);
    \draw[color=black,very thick] (-0.14,-0.24) .. controls (0.0,-0.26) .. (0.14,-0.24);
  }%
}
\newcommand{\KPLA}{%
  \KP{\draw[color=black,very thick] (-0.4,-0.10) .. controls (0.0,0.10) .. (0.4,-0.10);}%
}

\newcommand{\KPCA}{%
  \KP{  \draw[black,very thick] (-0.6,-0.2) .. controls (0.6,0.45) and (-0.6,0.45) .. (-0.18, 0.14); 
  \draw[black,very thick] (-0.09,0.08) .. controls (0.2,-0.14) .. (0.3,-0.2);
\!\!\!\!
}%
}
\newcommand{\KPCB}{%
  \KP{  \draw[black,very thick] (0.6,-0.2) .. controls (-0.6,0.45) and (0.6,0.45) .. (0.18, 0.14); 
  \draw[black,very thick] (0.09,0.08) .. controls (-0.2,-0.14) .. (-0.3,-0.2);
\!\!\!\!
}%
}

\newcommand*\wbar[1]{
  \hbox{ \kern-0.2em%
    \vbox{%
      \hrule height 0.5pt  
      \kern0.25ex
      \hbox{%
        \kern-0.10em
        \ensuremath{#1}%
        \kern-0.05em
      }%
    }%
  \kern0.05em}%
}

\newtheorem{theorem}{Theorem}[section]
\newtheorem{lemma}[theorem]{Lemma}
\newtheorem{proposition}[theorem]{Proposition}

\newtheorem{corollary}[theorem]{Corollary}
\newtheorem{conjecture}[theorem]{Conjecture}
\newtheorem*{theorem*}{Theorem}

\theoremstyle{definition}     
\newtheorem{definition}[theorem]{Definition}

\theoremstyle{remark}
\newtheorem{remark}[theorem]{Remark}
\newtheorem{example}[theorem]{Example}
\newtheorem{problem}[theorem]{Problem}

\title[The Jones-Krushkal polynomial and minimal diagrams of surface links]{The Jones-Krushkal polynomial and minimal  diagrams of surface links}
\author[H. U. Boden]{Hans U. Boden}
\address{Mathematics \& Statistics, McMaster University, Hamilton, Ontario}
\email{boden@mcmaster.ca}
\thanks{The first author was partially funded by the Natural Sciences and Engineering Research Council of Canada.}

\author[H. Karimi]{Homayun Karimi}
\address{Mathematics \& Statistics, McMaster University, Hamilton, Ontario}
\email{karimih@math.mcmaster.ca}


\date{\today}                 

\begin{document}

\begin{abstract}
We prove a Kauffman-Murasugi-Thistlethwaite theorem for alternating links in thickened surfaces. It states that any reduced alternating diagram of a link in a thickened surface has minimal crossing number, and any two reduced alternating diagrams of the same link have the same writhe. This result is proved more generally for link diagrams that are adequate, and the proof involves a two-variable generalization of the Jones polynomial for surface links defined by Krushkal. The main result is used to establish the first and second Tait conjectures for links in thickened surfaces and for virtual links. 
\end{abstract}

\maketitle
\section*{Introduction}
A link diagram is called \emph{alternating} if the crossings alternate between over and under crossing as one travels around any component; any link admitting such a diagram is called \emph{alternating}. In his early work of tabulating knots \cite{Tait}, Tait formulated several far-reaching conjectures which, when resolved 100 years later, effectively solved the classification problem for alternating knots and links. Recall that a link diagram is said to be \emph{reduced} if it does not contain any nugatory crossings. Tait's first conjecture states that any reduced alternating diagram of a link has minimal crossing number. His second states that any two such diagrams representing the same link have the same  writhe. His third conjecture, also known as the Tait flyping conjecture, asserts that any two reduced alternating diagrams for the same link are related by a sequence of flype moves (see Figure \ref{fig-flype}).  

Tait's first and second conjectures were settled through results of Kauffman, Murasugi, and Thistlethwaite, who each gave  an independent proof using the newly discovered Jones polynomial \cites{Kauffman-87, Murasugi-871, Thistlethwaite-87}. The Tait flyping conjecture was subsequently solved by Menasco and Thistlethwaite \cite{Tait3}, and taken together, the three Tait conjectures provide an algorithm for classifying alternating knots and links. A striking corollary is that the crossing number is additive under connected sum for alternating links. It remains a difficult open problem to prove this in general for arbitrary links in $S^3$.

Virtual knots were introduced by Kauffman in \cite{KVKT}, and they represent a natural generalization of classical knot theory to knots in thickened surfaces. Classical knots and links embed faithfully into virtual knot theory \cite{GPV}, and many  invariants from classical knot theory extend in a natural way. For instance, the Jones polynomial was extended to virtual links by Kauffman \cite{KVKT}, who noted the abundant supply of virtual knots with trivial Jones polynomial. (For classical knots, it is an open problem whether there is a nontrivial knot with trivial Jones polynomial.) Indeed, there exist alternating virtual knots $K$ with trivial Jones polynomial (and even trivial Khovanov homology  \cite{Karimi}). Consequently, the Jones polynomial is not sufficiently strong to prove the analogue of the Kauffman-Murasugi-Thistlethwaite theorem for virtual links
(see also \cite{Kamada-2004} and \cite{Dye-2017}).

The main result in this paper is a proof of the Kauffman-Murasugi-Thistlethwaite theorem for reduced alternating links in thickened surfaces. This result can be paraphrased as follows (see Theorem \ref{thm:KMT} and Corollary \ref{cor:KMT}):

\begin{theorem*} 
If $L$ is a non-split alternating link in a thickened surface $\Si \times I$, then any connected reduced alternating diagram for $L$ has minimal crossing number. Further, any two reduced alternating diagrams of $L$ have the same writhe.
\end{theorem*}

The theorem is established using a two-variable generalization of the Jones polynomial for links in thickened surfaces defined by Krushkal \cite{krushkal-2011}. The Jones-Krushkal polynomial is a homological refinement of the usual Jones polynomial in that it records the homological ranks of the states under restriction to the background surface. It is derived from Krushkal's extension of the Tutte polynomial to graphs in surfaces \cite{krushkal-2011}. The main result is proved more generally for diagrams that are adequate in a certain sense (see Definition \ref{defn:adequate}), and we show that every reduced alternating diagram of a link in a thickened surface is adequate (Proposition \ref{prop:adequate}). Further, if $L$ is a virtual link and $D$ is an alternating diagram for $L$, then we show that $L$ is split if and only if $D$ is a split diagram (Corollary \ref{cor:split}).  In addition, we prove the dual state lemma for links in thickened surfaces (Lemma \ref{lemma:dual-state}).  In the last section, Corollary \ref{cor:KMT} is applied to prove the Tait conjectures for virtual knots and links (see Theorem \ref{thm:Tait-knot} and \ref{thm:virtual-tait}).

In \cite{Adams}), Adams et al.~use geometric methods to prove minimality of reduced alternating diagrams of knots in thickened surfaces. In this paper, we generalize the results in \cite{Adams} to links in thickened surfaces admitting adequate diagrams. The proof relies on an analysis of the homological Kauffman bracket and Jones-Krushkal polynomial. These  invariants are closely related to the surface bracket polynomial studied in \cite{Dye-Kauffman-2005} and \cite{Manturov-2003a}. However, they exhibit different behavior in that they are not invariant under stabilization and destabilization. In \cite{krushkal-2011}, Krushkal shows that the Jones-Krushkal polynomial admits an interpretation in terms of the generalized Tutte polynomial of the associated Tait graph. This result is a generalization of Thistlethwaite's theorem \cite{Thistlethwaite-87}. In a similar vein, Chmutov and Voltz show how to relate the Jones polynomial of a checkerboard colorable virtual link with the Bollab\'as-Riordan polynomial of its Tait graph in \cite{Chmutov-Voltz} (see also \cite{Chmutov-Pak}).

We close this introduction with a brief synopsis of the contents of the rest of this paper.
In Section \ref{sec-1}, we review background material on links in thickened surfaces and virtual links. 
One result characterizes checkerboard colorable virtual links (Proposition \ref{prop:equiv}), and another characterizes alternating virtual links that are split (Corollary \ref{cor:split}). In Section \ref{sec-2}, we recall the definition of the homological Kauffman bracket $\lb \, \cdot \, \rb_\Si$ and show that it is invariant under regular isotopy. We prove that every reduced alternating link diagram is adequate (Proposition \ref{prop:adequate}) and establish the dual state lemma (Lemma \ref{lemma:dual-state}).  In Section \ref{sec-3}, we introduce the Jones-Krushkal polynomial $\wt{J}(t,z) \in \ZZ[t^{1/2}, t^{-1/2}, z]$ and show that it is an invariant of links in thickened surfaces up to isotopy and diffeomorphism. For checkerboard colorable links, we introduce the \emph{reduced} Jones-Krushkal polynomial $J(t,z)$ and give many sample calculations of $\wt{J}(t,z)$ and $ J(t,z)$. We prove that the Jones-Krushkal polynomial of $L$ is closely related to that of its mirror images $L^*$ and $L^\dag$
(Proposition \ref{star}).
  
Section \ref{sec-4} contains the proof of the main result, which is the Kauffman-Murasugi-Thistlethwaite theorem for reduced alternating link diagrams on surfaces (Corollary \ref{cor:KMT}). In Section \ref{sec-5}, we apply the main result to deduce the first and second Tait conjectures for virtual links (Theorem \ref{thm:virtual-tait}). Table \ref{table1} lists the reduced and unreduced Jones-Krushkal polynomials for all virtual knots with 3 crossings and all checkerboard colorable virtual knots with 4 crossings.

\bigskip
\noindent\emph{Notation.} Unless otherwise specified, all homology groups are taken with $\ZZ/2$ coefficients. Decimal numbers such as 3.5 and 4.98 refer to the virtual knots in Green's tabulation \cite{Green}.

\section{Virtual links and links in thickened surfaces} \label{sec-1}
In this section, we review the basic properties of virtual links and links in thickened surfaces.

\subsection{Virtual link diagrams}
A virtual link diagram is an immersion of  $m \geq 1$ circles in the plane with only double points, such that each double point is either classical (indicated by over- and under-crossings) or virtual (indicated by a circle around the double point). Two virtual link diagrams are said to be equivalent if they can be related by planar isotopies, Reidemeister moves,  and the \emph{detour move} shown in Figure \ref{Fig:detour}. An oriented virtual link $L$ includes a choice of orientation for each component of $L$, which is indicated by placing arrows on the components as in Figure \ref{fig:examples}.
      
\begin{figure}[!ht]
\centering\includegraphics[height=30mm]{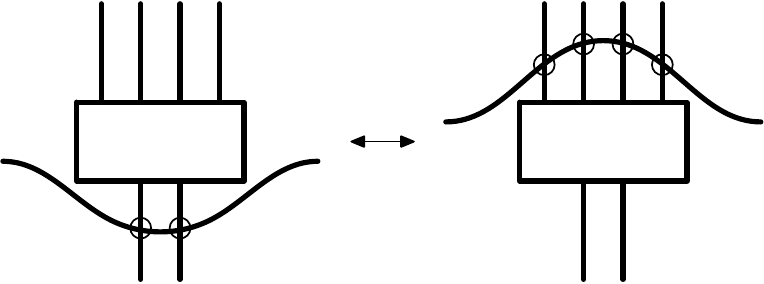}
\caption{The detour move.} \label{Fig:detour}
\end{figure}

Given a virtual link diagram $D$, the crossing number is denoted $c(D)$ and is defined to be the number of classical crossings of $D$. The \emph{crossing number} of a virtual link $L$ is the minimum crossing number $c(D)$ taken over all virtual link diagrams $D$ representing $L$. 

\begin{figure}[!ht]  
\centering
\includegraphics[height=24mm]{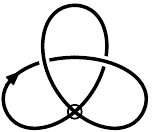}\hspace{.8cm}
\includegraphics[height=20mm]{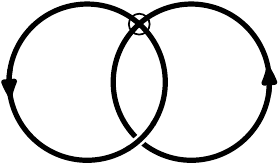}\hspace{.8cm}
\includegraphics[height=30mm]{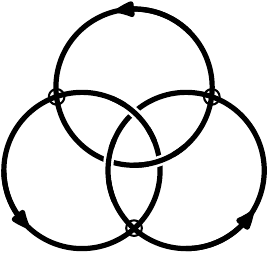} 
\caption{The virtual trefoil, Hopf link, and Borromean rings.} \label{fig:examples}
\end{figure}

Given an oriented virtual link, each classical crossing is either positive or negative, see Figure \ref{Fig:sign}. The writhe of the crossing is $\pm 1$ according to whether the crossing is positive or negative. The writhe of the diagram is the sum of the writhes of all its crossings.
 
\begin{definition} \label{defn:writhe}
For a virtual link diagram $D$, the \emph{writhe} of $D$, denoted $w(D)$ is defined as $n_{+}(D)-n_{-}(D)$, where $n_{+}(D)$ and $n_{-}(D)$ are the number of positive and negative crossings in $D$, respectively. 
\end{definition}
 
\begin{figure}[!ht]
\centering\includegraphics[height=25mm]{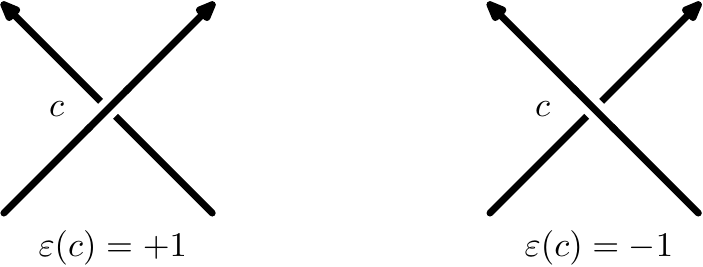}
\caption{A positive and a negative crossing.} \label{Fig:sign}
\end{figure}

\subsection{Links in thickened surfaces}
Virtual links can also be defined as equivalence classes of links in thickened surfaces.
Let  $I = [0,1]$ denote the unit interval and $\Si$ be a compact, connected, oriented surface.  A \emph{link in the thickened surface} $\Si \times I$ is an embedding $L \colon \bigsqcup_{i=1}^m S^1 \hookrightarrow \Si \times I$, considered up to isotopy and orientation preserving homeomorphisms  of the pair $(\Si \times I, \Si \times \{0\})$.

A \emph{surface link diagram} on $\Si$ is a tetravalent graph in $\Si$ whose vertices indicate over and under crossings in the usual way. Two surface link diagrams represent isotopic links if and only if they are equivalent by local Reidemeister moves. The writhe of a link diagram on a surface is defined in the same way as it is for virtual links (cf. Definition \ref{defn:writhe}).

Two link diagrams on $\Si$ are said to be \emph{regularly isotopic} if one can be obtained from the other by a sequence of moves that involve Reidemeister 2 and 3 moves and the writhe-preserving move in Figure \ref{Fig:rm-4}. Notice that the writhe is invariant under regular isotopy of links in $\Si \times I.$

\begin{figure}[!ht]
\centering\includegraphics[height=15mm]{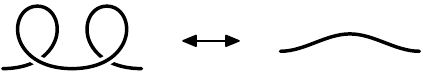}
\caption{Regular isotopy includes the above move together with Reidemeister 2 and 3 moves.} \label{Fig:rm-4}
\end{figure}

Let $p\co \Si \times I \to \Si$ be projection onto the first factor. The image
$p(L) \subset \Si$ is called the \emph{projection} of the link.  Using an isotopy, we can arrange that the projection is a regular immersion with finitely many double points.

Two links $L_0 \subset \Si_0\times I$ and $L_1 \subset \Si_1 \times I$ are said to be \emph{stably equivalent} if one is obtained from the other by a finite sequence of isotopies, diffeomorphisms, stabilizations, and destabilizations. Stabilization is the operation of adding a handle to $\Si$ to obtain a new surface $\Si'$, and destabilization is the opposite procedure. Specifically, if $D_0$ and $D_1$ are two disjoint disks in $\Si$ which are both disjoint from the image of $L$ under projection $\Si\times I\to \Si$, then $\Si'$ is the surface with $\genus(\Si') = \genus(\Si) + 1$ obtained by attaching an annulus $A = S^1 \times I$ to $\Si \sm (D_0 \cup D_1)$ so that $\partial A = \partial D_0 \cup \partial D_1$. This operation is referred to as \emph{stabilization}, and the opposite procedure is called \emph{destabilization}. It involves cutting along a vertical annulus in $\Si \times I$ disjoint from the link and attaching two 2-disks.

In \cite{Carter-Kamada-Saito}, Carter, Kamada, and Saito give a one-to-one correspondence between virtual links and stable equivalence classes of links in thickened surfaces. The next result is Kuperberg's theorem \cite{Kuperberg}.
 
\begin{theorem} \label{thm:kuperberg}
Every stable equivalence class of links in thickened surfaces has a unique irreducible representative.
\end{theorem}

Given a virtual link $L$, its virtual genus $g_v(L)$ is defined to be the genus of the surface of its unique irreducible representative. A virtual link $L$ is said to be \emph{classical} if  it has virtual genus $g_v(L)=0$. This is the case if and only if it can be represented by a virtual link diagram with no virtual crossings. For instance, the three virtual links in Figure \ref{fig:examples}  all have virtual genus equal to one and so are non-classical (see Figures \ref{fig:Hopf-in-torus}, \ref{fig:2-1-in-torus}, and \ref{fig:Borromean-in-torus}).

There is a construction, due to Kamada and Kamada \cite{KK00}, which associates to any virtual link diagram $D$ a ribbon graph on an oriented surface. The graph is a tetravalent graph representing the projection of $D$, and the surface $M_D$ has a handlebody decomposition with 0-handles being disk neighborhoods of each of the real crossings of $D$, and 1-handles for each of the arcs of $D$ from one crossing to the next. If $D$ has $n$ crossings, then $M_D$ has $n$ 0-handles and $2n$ 1-handles. Let $\Si_D$ denote the  closed oriented surface obtained by attaching disks to all the boundary components of $M_D$. A diagram $D$ for a virtual link $L$ is said to be a \emph{minimal genus diagram} if the genus of $\Si_D$ is equal to the virtual genus of $L$.

Notice that under this construction, the link diagram is cellularly embedded in $\Si_D$, namely the complement of its projection is a union of disks. By Theorem \ref{thm:kuperberg}, this will be true for any minimal genus diagram of a virtual link $L$.

\subsection{Alternating virtual links}

A virtual link diagram $D$ is said to be \emph{alternating} if, when traveling along the components, the classical crossings alternate from over to under when one disregards the virtual crossings. A virtual link $L$ is \emph{alternating} if it can be represented by an alternating virtual link diagram.

In a similar way, a surface link diagram $D$ on $\Si$ is said to be \emph{alternating} if the crossings alternate from over to under around any component of $D$. It follows that a virtual link is alternating if and only if it can be represented by an alternating surface link diagram.

\begin{definition}
Let $D$ be a surface link diagram on $\Si$. A crossing $c$ in $D$ is called \emph{nugatory} if we can find a simple closed curve in $\Si$ which separates $\Si$ and intersects $D$ only in the double point $c$. 
\end{definition}

\begin{remark} \label{rem:nug}
For classical link diagrams, nugatory crossings can always be removed by rotating one side of the diagram $180^\circ$ relative to the other. In contrast, for link diagrams on surfaces, nugatory crossings are not in general removable.
\end{remark} 

\begin{figure}[!ht]
\centering\includegraphics[height=32mm]{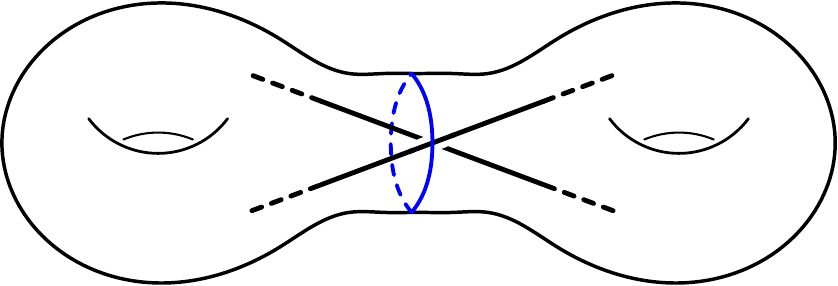}
\caption{A nugatory crossing.} \label{Fig-nugatory}
\end{figure}

\begin{definition} \label{defn:reduced}
A surface link diagram $D$ on $\Si$ is called \emph{reduced} if it is cellularly embedded and has no nugatory crossings.  
\end{definition}
\begin{remark} \label{rem:KK}
Note that, by the Kamada-Kamada construction, any virtual link can be realized by a cellularly embedded diagram on a surface. Thus, the first condition of Definition \ref{defn:reduced} can always be arranged for virtual links. 
\end{remark} 

\subsection{Checkerboard colorable links}
A surface link  diagram $D$ on $\Si$ is said to be \emph{checkerboard colorable} if the components of $\Si \sm D$ can be colored by two colors such that any two components of $\Si \sm D$ that share an edge have opposite colors. A  link in a thickened surface is checkerboard colorable if it can be represented by a checkerboard colorable surface link diagram. Likewise, a virtual link is checkerboard colorable if it admits a checkerboard colorable representative.
 
In \cite{Kamada-2002}, Kamada showed that every alternating virtual link is checkerboard colorable.
In fact, in \cite[Lemma 7]{Kamada-2002} she showed that a virtual link diagram is checkerboard colorable if and only if it can be transformed into an alternating diagram under crossing changes.

There are, however, subtle differences between the categories of virtual links and links in thickened surfaces. For instance, Figure \ref{Fig:trefoil} presents an alternating knot diagram on the torus which is not checkerboard colorable, hence Kamada's result is {\bf not} true for surface links. The failure stems from the fact that this diagram is not cellularly embedded. In particular, the diagram in Figure \ref{Fig:trefoil} is not minimal genus, and any vertical arc disjoint from the knot gives a destabilizing curve. Destabilization along this curve shows that this knot is stably equivalent to the classical trefoil, which of course is checkerboard colorable.

\begin{figure}[!ht]
\centering\includegraphics[height=32mm]{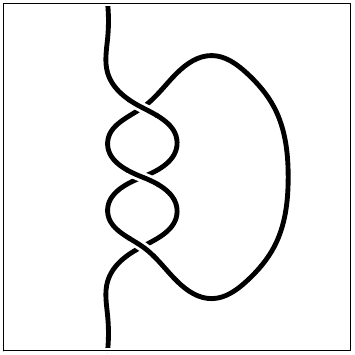}
\caption{An alternating knot diagram on the torus which is not checkerboard colorable.}
\label{Fig:trefoil}
\end{figure}

Several authors have used slightly different names for the notion of checkerboard colorability. For instance, in \cite{KNS-2002}, checkerboard colorable links are called \emph{normal}, and in \cite{Rushworth-2018}, checkerboard colorable diagrams called \emph{even}. In \cite{Boden-Gaudreau-Harper-2016}, checkerboard colorable links are called mod 2 almost classical links.

Suppose $D$ is a surface link diagram on $\Si$ which is cellularly embedded and checkerboard colorable, and fix a checkerboard coloring of the complementary regions of $D$ in $\Si$. The black regions determine a spanning surface for $L$ which is the union of disks and bands, with one disk for each black region and one half-twisted band for each crossing. 

The result is an unoriented surface $F$ embedded in $\Si \times I$ with boundary $\partial F =L.$  Associated to this surface is its Tait graph $\Ga$, which is a graph embedded in $\Si$ with one vertex for each black region and one edge for each crossing. There is an edge between two vertices for each crossing connecting the corresponding regions. In particular, if a black region has a self-abutting crossing, then its Tait graph $\Ga$ will contain a loop. 

The dual spanning surface for $L$ can be constructed by starting with the white regions and adding half-twisted bands for each crossing. Its Tait graph is defined similarly. The black and white Tait graphs are dual graphs in the surface $\Si$, and each of the checkerboard surfaces deformation retracts onto its Tait graph. The next result gives a useful characterization of checkerboard colorability for links in thickened surfaces.

\begin{proposition} \label{prop:equiv}
Given a link  $L \subset \Si \times I$ in a thickened surface, the following are equivalent:
\begin{enumerate}
\item[(i)] $L$ is checkerboard colorable.
\item[(ii)] $L$ is the boundary of an unoriented spanning surface $F \subset \Si \times I.$
\item[(iii)] $[L]=0$ in the homology group $H_1(\Si \times I; \ZZ/2).$
\end{enumerate}
\end{proposition}

\begin{proof}
If $L$ is checkerboard colorable, then an unoriented spanning surface is obtained by attaching one half-twisted band between two black regions for each crossing of $L$. This shows that (i)$\Rightarrow$(ii), and to see the reverse implication, suppose $F$ is a spanning surface for $L$, realized as a union of disks and bands in $\Si \times I$. Perform an isotopy to shrink the disks so their images under projection $p \co \Si \times I \to \Si$ are disjoint from one another and from each band. Thus, the projection, restricted to $F$, is an embedding except for band crossings. At each band crossing, we can attach a 1-handle so that the new surface is the black surface for a checkerboard coloring of the resulting diagram of $L$. Thus (ii)$\Rightarrow$(i).

The step (ii)$\Rightarrow$(iii) is obvious, and the reverse implication follows from a standard argument which is left to the reader. 
\end{proof} 


\subsection{Split virtual links}
In this section, we show that an alternating virtual link $L$ is split if and only if it is obviously split. In other words, $L$ is split if and only if every alternating diagram of $L$ is split. 

We begin by recalling an invariant of checkerboard colorable links called the link determinant.  Suppose $L$ is a virtual link that is represented by a checkerboard colorable diagram $D$ with $n$ crossings $\{c_1,\ldots, c_n \}$ and $m$ arcs $\{a_1,\ldots, a_{m}\}.$ Each arc $a_j$  starts at a classical undercrossing and goes to the next classical undercrossing, passing through any intermediate virtual crossings or overcrossings along the way. If $D$  has $k$ connected components, then $m = n+k-1$.

Define the $n\times m$ coloring matrix $B(D)$ so that its $ij$ entry is given by
\begin{eqnarray*}
b_{ij}(D) &=&\begin{cases}
2,&  \text{if $a_j$ is the overcrossing arc at $c_i$},\\
-1,& \text{if $a_j$ is one of the undercrossing arcs at $c_i$},\\
0,& \text{otherwise}.
\end{cases}
\end{eqnarray*}
In case $a_j$ is coincidentally the overcrossing arc and an undercrossing arc at $c_i$, then we set $b_{ij}(D)=1$. 

Note that the matrix $B(D)$ is the one obtained by specializing the Fox Jacobian matrix $A(D)$ at $t=-1$. Here, $A(D)$ is defined in terms of taking Fox derivatives of the Wirtinger presentation of the link group $G_D$ whose generators are given by the arcs $a_1,\ldots, a_m$ and relations are given by crossings $c_1,\ldots, c_n$ and applying the abelianization homomophism $G_L \to \lb t \rb, \ a_i \mapsto t$. 
For details, see \cite{Boden-Gaudreau-Harper-2016}.

Notice that the entries in each row of $B(D)$ sum to zero, therefore it has rank at most $n-1$. 
The next result is proved in \cite{Boden-Gaudreau-Harper-2016} .

\begin{proposition} \label{prop-alex-det}
Any two  $(n-1) \times (n-1)$ minors of $B(D)$ are equal up to sign. The absolute value of the minor is independent of the choice of checkerboard colorable diagram $D$. It defines an  invariant of checkerboard colorable links $L$ called the \emph{determinant} of $L$ and denoted $\det(L)$.
\end{proposition}

\begin{definition}\label{split-link}
A virtual link diagram $D$ is said to be a \emph{split diagram} if it is disconnected, and  a virtual link $L$ is \emph{split} if it can be represented by a split diagram. 
\end{definition}

\begin{proposition}\label{split det}
Suppose $L$ is a checkerboard colorable virtual link. If $L$ is split, then $\det(L)=0$. 
\end{proposition}

\begin{proof}
Suppose $D=D_1 \cup D_2$ is a split checkerboard colorable diagram for $L$.
In each row of the coloring matrix, the non-zero elements are either $2,-1,-1$ or $1,-1$. It follows the rows add up to zero. We consider a simple closed curve in the plane which separates $D$ into two parts. It follows that the coloring matrix $B=B(D)$ admits a $2 \times 2$ block decomposition of the form
$$B=\begin{bmatrix}
B_1 & 0 \\ 
0 & B_2
\end{bmatrix},$$
where $B_1$ and $B_2$ are the coloring matrices for $D_1$ and $D_2,$ respectively. Since $\det(B_1) =0=\det(B_2)$, it follows that  the matrix obtained by removing a row and column from $B$ also has determinant zero.      
\end{proof}

The next result is the virtual analogue of Bankwitz's theorem \cite{Bankwitz}. For a proof, see \cite{Karimi}. 

\begin{theorem}[Theorem 5.16, \cite{Karimi}] \label{link determinant}
Let $L$ be a  virtual link which is represented by a connected alternating diagram $D$ with $n\geq 2$ classical crossings. Suppose further that $D$ has no nugatory crossings. Then $\det(L)\geq n.$
\end{theorem}   

\begin{corollary} \label{cor:split}
Suppose a virtual link $L$ admits an alternating diagram $D$ without nugatory crossings. Then $L$ is a split virtual link if and only if  $D$ is a split diagram.
\end{corollary}

\begin{proof}
Clearly, if $D$ is a split diagram, then $L$ is a split link. Suppose then that $L$ admits an alternating diagram $D$ that is not split and has $n=n(D) >0$ crossings. (If $n=0$, then $D$ has one component and is an unknot diagram.) Theorem \ref{link determinant} implies that $\det(L)\geq n$. Hence $\det(L)\neq 0$, and Proposition \ref{split det} shows that $L$ is not split.
\end{proof}

\section{The homological Kauffman bracket} \label{sec-2}
In this section, we recall the definition of the homological Kauffman bracket from \cite{krushkal-2011}. It is defined for link diagrams in thickened surfaces and is an invariant of \emph{regular isotopy} of unoriented links. We introduce a notion of adequacy for link diagrams on surfaces and use the homological bracket to prove that adequate link diagrams have minimal crossing number.

\subsection{States and their homological rank}
Let $L$ be a link in $\Si \times I$ with surface link diagram $D$ on $\Si$. Suppose further that $D$ has $n$ crossings. For each crossing $c_i$ of $D$, there are two ways to resolve it. One is called the \emph{$A$-smoothing} and the other is the \emph{$B$-smoothing}, according to Figure \ref{resolving}.
 
\begin{figure}[!ht]
\centering\includegraphics[height=24mm]{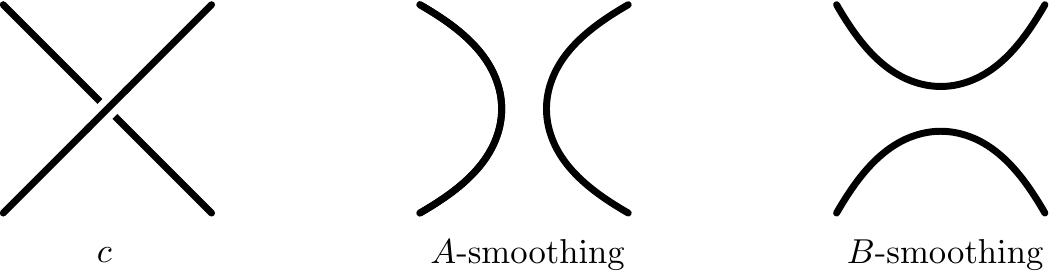}
\caption{The $A$- and $B$-smoothing of a crossing.}\label{resolving}
\end{figure}

A \emph{state} is a collection of simple closed curves on $\Si$ which results from smoothing each of the crossings of $D$. Thus, a state $S$ is just a link diagram on $\Si$ with no crossings. Since there are two ways to smooth each crossing, there are $2^n$ states. We will use $\sS = \sS(D)$ to denote the space of all states of $D$. Ordering the crossings $\{c_1,\ldots, c_n\}$ of $D$ in an arbitrary way, we can identify each state with a binary word $\ep_1 \ep_2 \cdots \ep_n$ of length $n$, where $\ep_i=0$ indicates an $A$-smoothing and $\ep_i = 1$ a $B$-smoothing at the crossing $c_i$.

Given a state $S \in \sS$, let $a(S)$ be the number of $A$-smoothings and $b(S)$ the number of $B$-smoothings, and let $|S|$ be the number of cycles in $S$. Define  
\begin{eqnarray*}
k(S) &=& \dim \left({\rm kernel}\left(i_* \colon H_1(S) \lto H_1(\Si)\right)\right), \\ 
r(S) &=& \dim \left({\rm image}\left(i_* \colon H_1(S) \lto H_1(\Si)\right)\right),
\end{eqnarray*}
where $i \colon S \to \Si$ is the inclusion map. We call $r(S)$ the \emph{homological rank}  of the state $S$, and we note that $k(S)+r(S) = b_1(S)=|S|.$ 

Since $\Si$ is a compact, closed, oriented surface, the intersection pairing on $H_1(\Si)$ is symplectic. A given collection of disjoint simple closed curves on $\Si$ must therefore map into an isotropic subspace of $H_1(\Si)$. It follows that the homological rank of any state $S$  satisfies $0 \leq r(S) \leq g,$ where $g$ is the genus of $\Si.$ 

The \emph{homological Kauffman bracket} is denoted $\lb \, \cdot \, \rb_\Si$ and defined by setting 
\begin{equation}
\lb D \rb_\Si =\sum_{S\in \sS} A^{(a(S)-b(S))}(-A^{-2}-A^2)^{k(S)} z^{r(S)}.
\label{eqn:bracket}
\end{equation}
Here, $z$ is a formal variable which keeps track of the homological rank of $S$. Upon setting $z = -A^{-2}-A^2$ and dividing one factor of $-A^{-2}-A^{2}$, one recovers the usual Kauffman bracket. 

\bigskip
The following lemmas study the effect of the various diagrammatic moves on the homological Kauffman bracket. These will be applied to show that it is invariant under regular isotopy of links in surfaces. The first is an immediate consequence of Equation  \eqref{eqn:bracket}, and the proof is left to the reader. 

In the first lemma, $\KPC$ denotes a simple closed curve on $\Si$.

\begin{lemma} \label{lemma:resolution}
The homological Kauffman bracket satisfies the following identities.
\begin{enumerate}
\item[(i)] If $\KPC$ is homologically trivial, then $\lb \KPC \rb_\Si= -A^2-A^{-2}.$
Otherwise, $\lb \KPC \rb_\Si= z.$
\item[(ii)]If $\KPC$ is homologically trivial, then $\lb  \KPC \sqcup L \rb_\Si = (-A^2-A^{-2}) \lb L \rb_\Si.$
\item[(iii)]
$\lb\KPX\rb_\Si=
  A\lb\KPA\rb_\Si + A^{-1} \lb \KPB \rb_\Si$.
\end{enumerate}
\end{lemma}

\begin{lemma} \label{lemma:RM-invariance1}
If a link diagram on a surface is changed by a Reidemeister type 1 move, then the homological Kauffman bracket changes as follows: 
\begin{equation} \label{eqn:kink}
\lb \!\!\!\! \KPCB \rb_\Si = -A^3 \lb\KPLA \rb_\Si \quad \text{ and } \quad \lb\KPCA \!\!\!\! \rb_\Si = -A^{-3} \lb\KPLA \rb_\Si.
\end{equation}
\end{lemma}
\begin{proof}
To see the first identity, apply Lemma \ref{lemma:resolution} (iii) to the left-hand side of Equation \eqref{eqn:kink} and simplify using Lemma \ref{lemma:resolution} (ii). The first equation follows immediately. The second identity follows by a similar argument; details are left to the reader.
\end{proof}

\begin{lemma} \label{lemma:RM-invariance2}
If a link diagram on a surface is changed by a Reidemeister type 2 or 3 move, then the homological Kauffman bracket is unchanged, i.e., we have
$$\text{(i)} \; \lb\KPD\rb_\Si = \lb\KPE\rb_\Si \quad \text{ and } \quad
\text{(ii)} \;\lb\KPXY\rb_\Si = \lb\KPXZ\rb_\Si$$
\end{lemma}
\begin{proof}
To prove (i), apply Lemma \ref{lemma:resolution} (iii) twice to the diagram on the left and simplify using Lemma \ref{lemma:resolution} (ii). The identity (i) follows. 

To prove (ii), apply Lemma \ref{lemma:resolution} (iii) to the lower crossing in the diagram on the left and simplify, using the fact that $\lb \, \cdot \, \rb_\Si$ is invariant under Reidemeister 2 moves.  The identity (ii) then follows.
\end{proof}

Lemma \ref{lemma:RM-invariance2} implies that the bracket $\lb D \rb_\Si$ is an invariant of unoriented links in $\Si \times I$ up to \emph{regular} isotopy, and in Definition \ref{defn:unreduced} we use Lemma \ref{lemma:RM-invariance1} to define
a normalization which is an invariant of oriented links in 
$\Si \times I$ up to isotopy. 

\begin{figure}[!ht]  
\centering
\includegraphics[height=26mm]{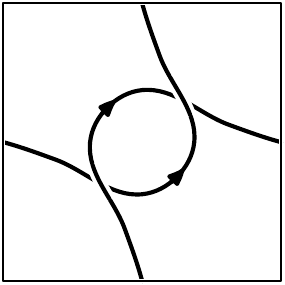}\hspace{.8cm}
\includegraphics[height=26mm]{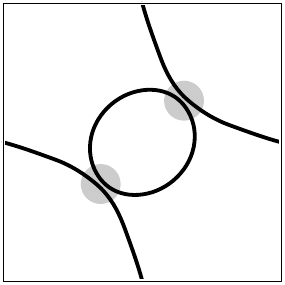}\hspace{.8cm}
\includegraphics[height=26mm]{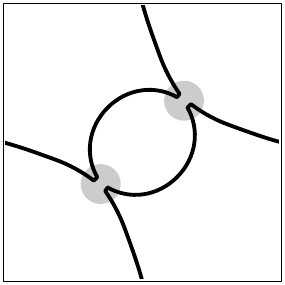} 
\caption{A minimal genus diagram of the virtual trefoil in the torus, and the states $S_A$ and $S_B$.} \label{fig:2-1-in-torus}
\end{figure}

\begin{example} \label{example-2-1-a}
The virtual trefoil $K$ (see Figure \ref{fig:examples}) admits a minimal genus diagram $D$ on the torus $T$, which has two crossings.  The diagram $D$ is depicted in  Figure \ref{fig:2-1-in-torus}, along with the state $S_A$ of all $A$ smoothings and the state $S_B$ of all $B$-smoothings. Clearly $|S_A|=2, r(S_A)=1$ and $|S_B|=1, r(S_B)=1$. One can further show that the other two states $AB, BA$ have $|S|=1$ and $r(S)=1.$ Thus,  $\lb D \rb_T = A^2(-A^2-A^{-2})z + 2 z +A^{-2}z$. 
\end{example}
Notice that the cube of resolutions for this knot has single cycle smoothings. These occur whenever there are two states $S,S'$ with $|S|=|S'|$ which are identical everywhere except one crossing. For checkerboard colorable diagrams, one can show that $|S|=|S'| \pm 1$ whenever $S,S'$ are two states that differ only at one crossing (for a proof, see \cite{Rushworth-2018} or \cite[Proposition 6.14]{Karimi}). Therefore, the cube of resolutions of a checkerboard colorable diagram never has any single cycle smoothings.

\subsection{Adequate diagrams}

Next, we  introduce the notions of $A$-adequate and $B$-adequate for link diagrams on a surface. In the following, for a given link diagram $D$ on a surface, let $S_{A}$ denote the all $A$-smoothing state and $S_{B}$ the all $B$-smoothing state. 

We take a moment to review the state-sum formulation for the homological Kauffman bracket. Given a link diagram $D$ on $\Si$ and a state $S \in \sS(D)$, let
\begin{equation} \label{eqn:single-bracket}
 \lb D\,|\,S \rb_\Si =A^{(a(S)-b(S))}(-A^{-2}-A^{2})^{k(S)}z^{r(S)}. 
\end{equation}
Then  we can write
\begin{equation} \label{eqn:state-sum}
\lb D \rb_\Si = \sum_{S\in \sS} \lb D\,|\,S \rb_\Si.
\end{equation}

\begin{definition} \label{defn:adequate}
The diagram $D$ is called \emph{$A$-adequate}, if for any state $S'$ with exactly one $B$-smoothing, we have $k(S') \leq k(S_A)$. The diagram $D$ is called \emph{$B$-adequate} if, for any state $S'$ with exactly one $A$-smoothing, we have $k(S') \leq k(S_B)$. A diagram is called \emph{adequate} if it is both $A$- and $B$-adequate.
\end{definition}  

Recall that for a classical link, a link diagram $D$ is ``plus-adequate'' if the all $A$-smoothing state $S_A$ does not contain any self-abutting cycles, and it is ``minus-adequate'' if the same holds for the all $B$-smoothing state $S_B$ \cite[Definition 5.2]{Lickorish}. Thus, if a diagram is plus-adequate then it is $A$-adequate, and if it is minus-adequate then it is $B$-adequate.

However, a diagram can be $A$-adequate without being plus-adequate, and it can be $B$-adequate without being minus-adequate. Indeed, our notion of adequacy is less restrictive because it allows self-abutting cycles in $S_A$ provided that $k(S')\leq k(S_A)$ for the new state $S'$ obtained by switching the smoothing. In case $|S'| = |S_A|+1$, this is equivalent to the requirement that $r(S') = r(S_A)+1.$ There is a similar interpretation for $B$-adequacy.

In \cite[p.1089]{Kamada-2004}, Kamada defines a virtual link diagram to be \emph{proper} if four distinct regions of the complement meet at every crossing. Any virtual link diagram that is proper is automatically adequate according to Definition \ref{defn:adequate}, but the converse is not true in general. Indeed, in Proposition \ref{prop:adequate}, we will show that all reduced alternating link diagrams on surfaces are adequate, whereas most alternating knot diagrams on surfaces are not proper.

We use $d_{\max}$ and $d_{\min}$ to denote the maximal and minimal degree in the variable $A$.  For example, $d_{\max} (\lb D\,|\,S \rb_\Si) = a(S)-b(S)+2k(S)$ and $d_{\min} (\lb D\,|\,S \rb_\Si) = a(S)-b(S)-2k(S)$. (Note that the homological variable $z$ is disregarded in degree considerations.)

\begin{lemma} \label{lemma3.6}
If $D$ is a surface link diagram on $\Si$ with $n$ crossings, then 
\begin{itemize}
\item[(i)] $d_{\max}(\lb D \rb_\Si )\leq n+2k(S_A)$, with equality if $D$ is $A$-adequate, 
\item[(ii)] $d_{\min}(\lb D \rb_\Si )\geq -n-2k(S_B)$, with equality if $D$ is $B$-adequate.
\end{itemize}
\end{lemma}

\begin{proof}
Suppose $S$ is a state for $D$ with an $A$-smoothing at a given crossing but otherwise arbitrary, and let $S'$ be the state obtained by switching it to a $B$-smoothing at the given crossing. Clearly, $a(S')=a(S)-1$ and $b(S')=b(S)+1$. Switching the crossing produces a cobordism from $S$ to $S'$, and there are three possibilities: (i) two cycles in $S$ join to form one cycle in $S'$, (ii) one cycle in $S$ splits to form two cycles in $S'$, or (iii) switching from $S$ to $S'$ involves a single cycle smoothing (see Figure \ref{Fig-pairofpants}). Notice that $|S'| =|S|-1$, $|S'| =|S|+1$, or $|S'| =|S|$ in cases (i), (ii), or (iii), respectively.

\begin{figure}[!ht]
\centering\includegraphics[height=32mm]{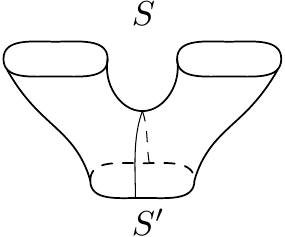}
\hspace{1cm}
\includegraphics[height=32mm]{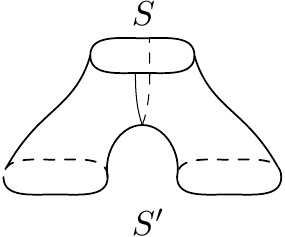}
\hspace{1.3cm}
\includegraphics[height=32mm]{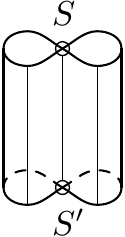}
\caption{The three types of cobordisms from $S$ to $S'$ include a fusion (left), fission (middle), and single cycle smoothing (right).} \label{Fig-pairofpants}
\end{figure}
 
Further, in case (i), either $r(S') =r(S)$ and $k(S') = k(S)-1$ or $r(S') =r(S)-1$ and $k(S') = k(S)$; and in case (ii), either $r(S') =r(S)$ and $k(S') = k(S)+1$ or $r(S') =r(S)+1$ and $k(S') = k(S)$. In case (iii), either $r(S') =r(S)$ and $k(S') = k(S)$ or $r(S') =r(S) \pm 1$ and $k(S') = k(S) \mp 1$. Since $a(S')-b(S')=a(S)-b(S)-2$ and $k(S') \leq k(S)+1$ in all three cases, we conclude that
\begin{equation} \label{eqn:repeat}
d_{\max}\left(\lb D \,|\,S' \rb_\Si\right) \leq d_{\max} \left(\lb D \,|\,S \rb_\Si\right).
\end{equation}

Clearly $d_{\max}(\lb D \,|\,S_A \rb_\Si )= n+2k(S_A)$ and $d_{\min}(\lb D \,|\,S_B \rb_\Si )= -n-2k(S_B)$. Since any state is obtained from $S_A$ by switching smoothings at a finite set of crossings, repeated application of Equation \eqref{eqn:repeat} gives that $d_{\max}(\lb D \,|\,S \rb_\Si )\leq  d_{\max}(\lb D \,|\,S_A \rb_\Si)$, and the inequality (i) follows.

Now suppose that $D$ is $A$-adequate and $S$ is a state with exactly one $B$-smoothing. Then $A$-adequacy implies that $k(S)\leq k(S_A)$. Since $a(S)-b(S)=n-2$, it follows that
\begin{equation} \label{eqn:top}
d_{\max}(\lb D \,|\,S \rb_\Si)  \leq  d_{\max}(\lb D \,|\,S_A \rb_\Si)-2.
\end{equation}
Any state $S'$ with two or more $B$-smoothings is obtained from a state $S$ with exactly one $B$-smoothing by switching the smoothings at the remaining crossings. Therefore, by Equations \eqref{eqn:repeat} and \eqref{eqn:top}, we find that  
$$d_{\max}(\lb D \,|\,S' \rb_\Si)  \leq d_{\max}(\lb D \,|\,S \rb_\Si) \leq  d_{\max}(\lb D \,|\,S_A \rb_\Si)-2.$$ Thus $d_{\max}(\lb D \rb_\Si) = n+2k(S_A),$ and this completes the proof of (i).

Statement (ii) follows by a similar argument. Alternatively, one can deduce (ii) directly from (i) using the observation that a diagram is $A$-adequate if and only if its mirror image is $B$-adequate. 
\end{proof}

Define the span of the homological Kauffman bracket by setting 
$$\spn (\lb D \rb_\Si) =  d_{\max}(\lb D \rb_\Si) - d_{\min}(\lb D \rb_\Si).$$
By Lemma \ref{lemma:RM-invariance2}, the homological Kauffman bracket is invariant under the second and third Reidemeister moves. Lemma \ref{lemma:RM-invariance1} implies that   $\spn (\lb D \rb_\Si)$ is also invariant under the first Reidemeister move. Therefore, it gives an invariant of the underlying link.

\begin{corollary} \label{cor-adequate}
If $D$ is a link diagram with $n$ crossings on a surface $\Si$, then 
$$\spn(\lb D\rb_\Si) \leq 2n+2k(S_{A})+2k(S_{B}),$$
with equality if $D$ is adequate.
\end{corollary}

\begin{proposition} \label{prop:adequate}
Any reduced alternating diagram $D$ for a link in a thickened surface $\Si\times I$ is adequate. 
\end{proposition}

\begin{proof}
Any reduced alternating link diagram on a surface is checkerboard colorable, and one can choose the coloring so that the white regions are enclosed by the cycles of $S_A$ and the black regions by the cycles of $S_B.$ Since each cycle in $S_A$ bounds a white region, it follows that $S_A$ is homologically trivial. Thus $r(S_A)=0$ and $k(S_A)=b_1(S_A)=|S_A|.$  Similarly, since each cycle in $S_B$ bounds a black region, $S_B$ is also  homologically trivial. Thus $r(S_B)=0$ and $k(S_B)=b_1(S_B)=|S_B|.$

We will now show that $D$ is $A$-adequate. Suppose $S$ is a state with exactly one $B$-smoothing. Then $|S| = |S_A| \pm 1.$ (Since $D$ is checkerboard colorable, there are no single cycle smoothings.) If $|S| = |S_A|-1,$ then $r(S) \leq r(S_A) =0$.  Hence $r(S)=0$ and $k(S)=|S|=|S_A|-1 = k(S_A)-1$ as required. Otherwise, if $|S| = |S_A|+1,$ then we claim that $r(S)=1$ and $k(S)=|S|-1=|S_A|=k(S_A)$.

To prove this claim, consider a self-abutting cycle of $S_A.$ This happens only for crossings of $D$ where a white region meets itself. Such a crossing gives rise to a loop $\ga$ in the associated Tait graph $\Ga$. We view $\Ga$ as a graph embedded in $\Si.$ Since $D$ is reduced, it contains no nugatory crossings, hence the loop $\ga$ must be a non-separating curve on $\Si.$ This implies the loop is homologically nontrivial in $\Si,$ namely $[\ga] \neq 0$ as an element in $H_1(\Si).$ Since each cycle in $S_A$ is homologically trivial, the two new  cycles formed by switching the smoothing must both carry the homology class $[\ga]$. In particular, this shows that $r(S)=1$, and it follows that $k(S)=|S|-r(S)=|S_A| = k(S_A).$ This completes the proof that $D$ is $A$-adequate.

The same argument applied to the mirror image of $D$ shows that the diagram $D$ is $B$-adequate. 
\end{proof}

\begin{theorem} \label{thm:span}
Suppose $L$ is a link in $\Si \times I$ admitting a connected reduced alternating diagram $D$ on $\Si$. Then $\spn (\lb D \rb_\Si) = 4n-4g+4$, where $n$ is the number of crossings of $D$ and $g$ is the genus of $\Si$. 
\end{theorem}

\begin{proof}
Since $D$ is a reduced alternating diagram, it is checkerboard colorable. Further, we can choose the coloring so that the cycles of $S_A$ are the boundaries of the white disks and the cycles of $S_B$ are the boundaries of the black disks. Thus $|S_A|$ is the number of white disks and $|S_{B}|$ is the number of black disks.  The diagram $D$ gives a handlebody decomposition of $\Si$, and using that to compute  the Euler characteristic, we find that $\chi(\Si) = n -2n + |S_A| + |S_B|.$ It follows that $|S_A| + |S_B| = n +2 -2g.$

By Proposition \ref{prop:adequate}, $D$ is adequate, and
 Corollary \ref{cor-adequate} applies to show that
\begin{eqnarray*}
\spn (\lb D \rb_\Si) &=& 2n+2(k(S_{A})+k(S_{B})), \\
&=& 2n+2(|S_{A}|+|S_{B}|), \\
&=& 4n - 4g+4. \qedhere
\end{eqnarray*}
\end{proof}

\subsection{The dual state lemma}
The next result is the analogue of the dual state lemma for surface link diagrams. Note that for a given state $S$, the dual state, denoted $S^\vee$, is given by performing the opposite smoothing at each crossing. For classical links, the dual state lemma was proved by Kauffman and Murasugi \cites{Kauffman-87, Murasugi-871}. Our proof is based on the one given by Turaev \cite[\S2]{Turaev-1987}.  

\begin{lemma} \label{lemma:dual-state} 
Let $D$ be a connected link diagram on a surface $\Si$ with genus $g$, 
and suppose $D$ has $n$ crossings.  For any state $S$ with dual state $S^\vee$, we have:
\begin{itemize}
\item[(a)] $|S| + |S^\vee| \leq n+2$.
\item[(b)] $k(S) + k(S^\vee) \leq n+2-2g$ provided that $D$ is cellularly embedded. 
\end{itemize}
\end{lemma}

\begin{proof}
Assume the diagram $D$ lies in $\Si \times \{1/2\}$ and that its double points (or crossings) have been labeled $c_1,\ldots, c_n$.
 
Given a state $S$ for $D$, we construct a compact surface $M_S$ with boundary $\partial M_S =  S \cup S^\vee$ embedded in $\Si \times I$. The surface is a union of disks and bands, and it has one disk for each crossing and one band for each edge of $D$. Specifically, if there is an edge of $D$ connecting $c_i$ to $c_j$, then there is a band of $M_S$ connecting the disk at $c_i$ to the one at $c_j$. Note that we are not excluding loops, which occur if $c_i=c_j.$ 

The disks of $M_S$ are assumed to lie in $\Si \times \{1/2\}$ and to be pairwise disjoint. The bands of $M_S$ retract to the corresponding edge of $D$, but they sometimes include a half-twist.  The bands without twists are assumed to lie in $\Si \times \{1/2\}$, and the bands with a half-twist are assumed to lie in a small neighborhood of the associated edge of $D$. (The direction of the half-twist is immaterial.) From this description, it is clear that there is a deformation retract from $M_S$ to the diagram $D$.

\begin{figure}[!ht]  
\centering
\includegraphics[height=25mm]{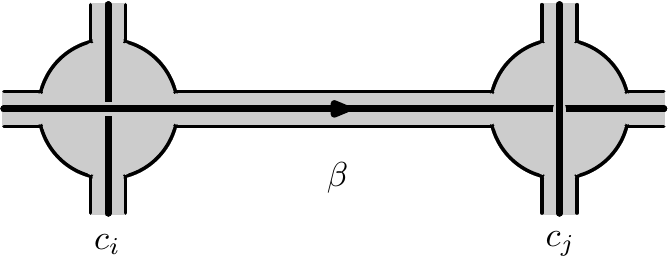}\hspace{.5cm}
\includegraphics[height=25mm]{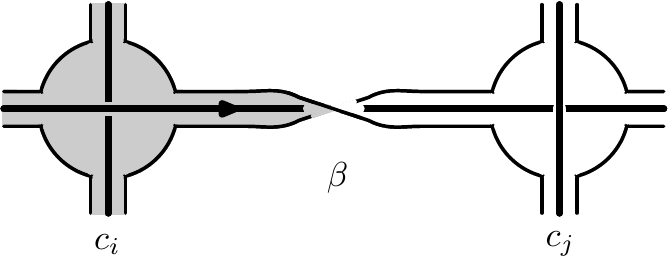}
\caption{The band $\be$ connecting the disk at $c_i$ to the disk at $c_j$. Since the crossings are opposite (over/under), the band will be untwisted if the smoothings are the same ($AA$ or $BB$) and twisted if the smoothings are opposite ($AB$ or $BA$).} \label{fig:band}
\end{figure}

Let $\be$ be a band connecting the disk at $c_i$ to the disk at $c_j$, and we discuss now whether or not $\be$ is flat or twisted as in Figure \ref{fig:band}. It connects one of outgoing arcs of $c_i$ to one of the incoming arcs of $c_j$. There are four possibilities, according to whether the outgoing arc from $c_i$ is an overcrossing or an undercrossing arc, and whether the incoming arc to $c_j$ is an overcrossing or an undercrossing arc. There are also four possibilities according to the smoothings the state $S$ specifies at $c_i$ and $c_j$, which is one of $\{AA, AB, BA, BB\}$. The band $\be$ will be untwisted if the arcs are opposite (over/under or under/over) and the smoothings are the same ($AA$ or $BB$), or if arcs are the same (over/over or under/under) and the smoothings are the opposite ($AB$ or $BA$). Otherwise, the band $\be$ includes a half-twist (see Figure \ref{fig:band}).

This same prescription applies in case $\be$ is a loop. In that case, the smoothings at $c_i=c_j$ are necessarily the same, thus the band $\be$ will be flat if the outgoing arc at $c_i$ connects to the other incoming arc at $90^\circ$, and it will be half-twisted if it connects to the opposite incoming arc (see Figure \ref{fig:band-loop}).

\begin{figure}[!ht]  
\centering
\includegraphics[height=40mm]{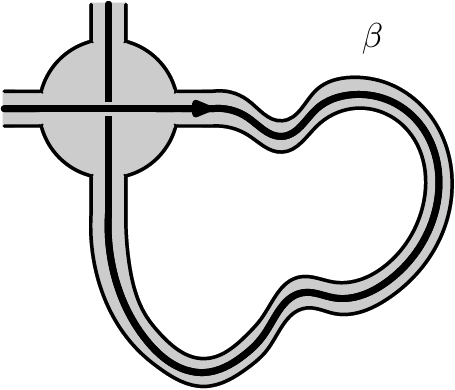}\hspace{1.5cm}
\includegraphics[height=40mm]{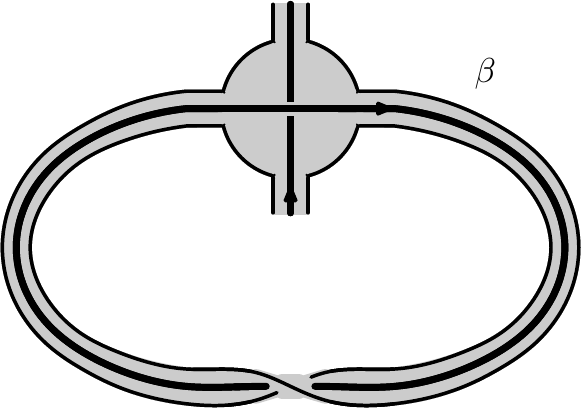}
\caption{A flat band and a half-twisted band.} \label{fig:band-loop}
\end{figure}

When the surface $M_S$ is defined this way, it follows that $\partial M_S = S \cup S^\vee$. Consider the commutative diagram:
\begin{equation}\label{eqn:diagram}
\xymatrix{
  H_2(M_S,\partial M_S) \ar[r]   & H_1(\partial M_S)  \ar[r] \ar[d]^{i_*} & H_1(M_S) \ar[r] \ar[d]^{j_*} & \cdots \\
  & H_1(\Si) \ar[r]^{=} & H_1(\Si)   
}
\end{equation}
In the above, all homology groups are taken with $\ZZ/2$ coefficients, and the top row is the long exact sequence in homology of the pair $(M_S,\partial M_S)$, and the two vertical maps are induced by $M_S \stackrel{j}{\hookrightarrow} \Si \times I \stackrel{p}{\to} \Si$ and $\partial M_S \stackrel{i}{\hookrightarrow} \Si \times I \stackrel{p}{\to} \Si$.

Since $M_S$ is connected and  has $\chi(M_S) = n-2n = -n$, it follows that $b_1(M_S) = b_0(M_S) - \chi(M_S) = n+1.$ Further, $b_2(M_S,\partial M_S)=1$, thus
$$|S| +|S^\vee| = b_1(\partial M_S) \leq  b_1(M_S) +b_2(M_S,\partial M_S) = n+2,$$ 
which proves part (a). 

If $D$ is cellularly embedded, then the map $H_1(D) \to H_1(\Si)$ induced by inclusion is surjective. However, since $M_S$ deformation retracts to $D$, it follows that the map $j_*$ in \eqref{eqn:diagram} is also surjective. Thus $\dim(\ker j_*) = b_1(M_S) - 2g = n+1-2g.$

By commutativity of  \eqref{eqn:diagram}, this implies that
$$k(S) +k(S^\vee) = \dim(\ker i_*) \leq  \dim(\ker j_*) +b_2(M_S,\partial M_S) = n+2 - 2g,$$ 
which proves part (b). 
\end{proof}
 
\section{The Jones-Krushkal polynomial} \label{sec-3}
In this section, we recall the Jones-Krushkal polynomial, which is a two-variable Jones-type polynomial associated to oriented links in thickened surfaces  and defined in terms of the homological Kauffman bracket \cite{krushkal-2011}. We show that this polynomial, or rather its reduction, has a special form when the link $L$ is checkerboard colorable.  We provide many sample calculations, and we prove a result that describes its behavior under horizontal and vertical mirror symmetry.

\subsection{The Jones-Krushkal polynomial}

\begin{definition} \label{defn:unreduced}
For an oriented link $L$ in a thickened surface $\Si \times I$ with link diagram $D$, the (unreduced) Jones-Krushkal polynomial is given by setting $\wt{J}_L(t, z)= \left[ (-A)^{-3w(D)} \lb D \rb_\Si \right]_{A=t^{-1/4}}$. Thus, we have 
$$\wt{J}_L(t, z)= (-1)^{w(D)} t^{3w(D)/4} \sum_{S \in \sS} t^{(b(S)-a(S))/4}(-t^{-1/2}-t^{1/2})^{k(S)} z^{r(S)}.$$ 
\end{definition}

The usual Jones polynomial is defined similarly:
$$V_L(t)= (-1)^{w(D)} t^{3w(D)/4} \sum_{S \in \sS} t^{(b(S)-a(S))/4}(-t^{-1/2}-t^{1/2})^{|S|-1}.$$
Since $|S|=k(S) + r(S),$ it is clear that one can recover the usual Jones polynomial from $\wt{J}_L(t, z)$ by setting $z=-t^{-1/2}-t^{1/2}$ and dividing one $(-t^{-1/2}-t^{1/2})$ factor out.

The factor $t^{3w(D)/4}$ is chosen so that the right hand side of the above equation is invariant under all three Reidemeister moves. Lemma \ref{lemma:RM-invariance1} implies that the polynomial $\wt{J}_L(t,z)$ is invariant under isotopy of links in $\Si \times I$. It is also an invariant of diffeomorphism of the pair $(\Si \times I, \Si \times \{0\}).$  By Kuperberg's theorem, we can obtain an invariant of virtual links by calculating the polynomial on a minimal genus representative.

The next lemma shows that the Jones-Krushkal polynomial is a Laurent polynomial in $t^{1/2}.$  

\begin{lemma} \label{lemma:f-poly}
If $L$ is an oriented link in $\Si \times I$, then 
$\wt{J}_L(t,z) \in \ZZ[t^{1/2}, t^{-1/2}, z].$ 
\end{lemma}
\begin{proof}
Equivalently, we claim that, for any surface link diagram $D$ on $\Si$, the normalized Kauffman bracket $(-A)^{-3w(D)} \lb D \rb_\Si$ lies in $\ZZ[A^2, A^{-2}, z]$.

Note that the claim, once proved, implies the lemma. Note further that Equation \ref{eqn:single-bracket} implies that the terms in $\lb D\,|\, S \rb_\Si$ all have the same $A$-degree modulo 4 for any state $S \in \sS(D)$.  For two states $S_1, S_2 \in \sS(D)$, we have $a(S_1)-b(S_1) \equiv a(S_2)-b(S_2) \mod 2$. Hence Equation \ref{eqn:single-bracket} implies that the terms in $\lb D\,|\, S_1 \rb_\Si$ and in $\lb D\,|\, S_2 \rb_\Si$ have the same $A$-degree modulo 2. Thus, the claim will follow once it has been verified for any one state $S \in \sS(D)$.

We claim that $(-A)^{-3w(D)} \lb D\,|\, S_\si \rb_\Si \in \ZZ[A^2, A^{-2}, z]$, where $S_\si$ is the \emph{Seifert} state. This is the state with all oriented smoothings (see Figure \ref{Fig:oriented}). (For classical links, $S_\si$ coincides with the one produced by Seifert's algorithm.) As in Figure \ref{Fig:oriented}, $S_\si$ has $A$-smoothings at the positive crossings and $B$-smoothings at the negative crossings. Thus $a(S_\si) - b(S_\si) = w(D)$.

\begin{figure}[!ht]
\centering\includegraphics[height=25mm]{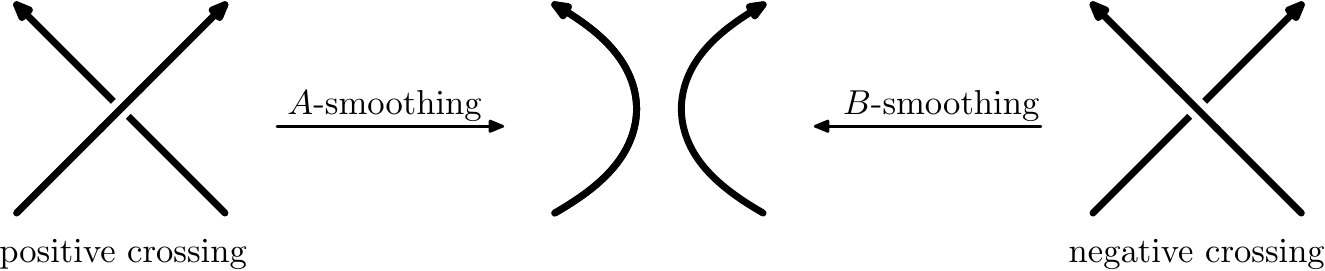}
\caption{Oriented smoothings at positive and negative crossings.}\label{Fig:oriented}
\end{figure}

To complete the proof, we apply Equation \ref{eqn:single-bracket} one more time to see that
\begin{eqnarray*}
(-A)^{-3w(D)} \lb D \,|\, S_\si \rb_\Si &=& (-A)^{-3w(D)}\left(A^{w(D)}(-A^{-2}-A^2)^{k(S_\si)}z^{r(S_\si)} \right),\\
&=& (-1)^{w(D)} A^{-2w(D)}(-A^{-2}-A^2)^{k(S_\si)}z^{r(S_\si)} \in \ZZ[A^2,A^{-2}, z]. \qedhere
\end{eqnarray*}
\end{proof}

For a link $L \subset \Si \times I$, where $\Si$ has genus $g$, it follows that $0 \leq r(S) \leq g$ for all states $S \in \sS.$ Thus, we can write $\wt{J}_L(t, z)= \sum_{i=0}^g \wt{\Phi}_i(t) z^i$, where $\wt{\Phi}_i(t) \in \ZZ[t^{-1/2}, t^{1/2}]$ for $i=0,\ldots, g$ by Lemma \ref{lemma:f-poly}.

\begin{proposition} \label{prop:non-check}
If $L$ is a link in $\Si \times I$ which is not checkerboard colorable, then $\wt{\Phi}_0(t) =0$. Thus $\wt{J}_L(t,z) = z J'(t,z)$,  where $J'(t,z) =\sum_{i=1}^g \wt{\Phi}_i(t) z^{i-1}.$ 

If, in addition, $L$ is a link in a thickened torus, then $\wt{J}(t,z)=zV_L(t)$ and so is completely determined by the usual Jones polynomial.
\end{proposition}

\begin{proof}
If $L$ is not checkerboard colorable, then $[L]$ is nontrivial as an element in $H_1(\Si)$. The same is true for any state $S$, since $S$ is homologous to $L$.  Thus  $r(S) \geq 1$ for all states, which implies that $\Phi_0(t) = 0$. If $L$ is a link in the thickened torus, then it follows that $r(S)=1$ for all states, thus $\wt{J}(t,z)=zV_L(t)$ as claimed.
\end{proof}

\begin{example} \label{example-2-1-b}
Let $K$ be the virtual trefoil (see Figures \ref{fig:examples} and \ref{fig:2-1-in-torus}). In Example \ref{example-2-1-a}, we showed that its diagram has homological Kauffman bracket $\lb D \rb_T = A^2(-A^2-A^{-2})z + 2 z +A^{-2}z$. Since this diagram has writhe $w(D)=-2$, it follows that
\begin{equation} \label{eqn:virtual-trefoil}
\wt{J}_{K}(t,z) = z\left(- t^{-5/2}+t^{-3/2}+  t^{-1} \right).
\end{equation}
Since $K$ is not checkerboard colorable and has virtual genus one, Equation \eqref{eqn:virtual-trefoil} can also be deduced from  Proposition \ref{prop:non-check} and the fact that $V_K(t)=- t^{-5/2}+t^{-3/2}+  t^{-1}$.
\end{example}

\begin{figure}[!ht]  
\centering
\includegraphics[height=26mm]{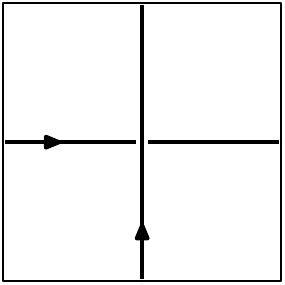}\hspace{.8cm}
\includegraphics[height=26mm]{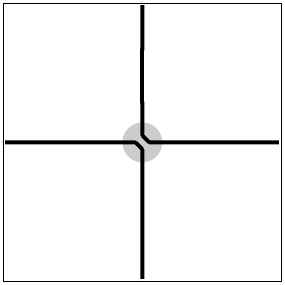}\hspace{.8cm}
\includegraphics[height=26mm]{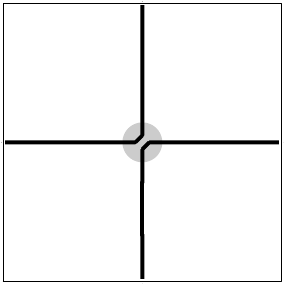} 
\caption{A minimal genus diagram of the virtual Hopf link in the torus, and the states $S_A$ and $S_B$.} \label{fig:Hopf-in-torus}
\end{figure}

\begin{example} \label{example-virtual-Hopf}
The virtual Hopf link $L$ (see Figure \ref{fig:examples}) admits a minimal genus diagram $D$ on the torus $T$ which has one crossing. The diagram $D$, along with the states $S_A$ and $S_B$, are depicted in Figure \ref{fig:2-1-in-torus}. It has $|S_A| = 1=|S_B|$ and $r(S_A) = 1 = r(S_B),$ hence $\lb D \rb_T = Az  +A^{-1}z$. Thus, the Jones-Krushkal polynomial of the virtual Hopf  link is $\wt{J}_{L}(t,z) = z\left(-t^{-1}-t^{-1/2}\right)$. Since the virtual Hopf link is not checkerboard colorable and has virtual genus one, this also follows from Proposition \ref{prop:non-check} and the fact that $V_{L}(t) = -t^{-1}-t^{-1/2}$.
\end{example}

\subsection{The reduced Jones-Krushkal polynomial}
In this section, we introduce a reduction of Jones-Krushkal polynomial for checkerboard colorable links in thickened surfaces.  

Suppose $L$ is a link in $\Si \times I$ represented by a checkerboard colorable link diagram $D$ on $\Si$. Since $L$ is homologically trivial as an element in $H_1(\Si \times I; \ZZ/2),$ it follows that $k(S)\geq 1$ for each state $S$. 

\begin{definition}
Let $L$ be an oriented, checkerboard colorable link in $\Si \times I$ and $D$ a diagram on $\Si$ representing $L$. The reduced Jones-Krushkal polynomial is defined by setting
$$J_L(t, z)= (-1)^{w(D)} t^{3w(D)/4} \sum_{S \in \sS} t^{(b(S)-a(S))/4}(-t^{-1/2}-t^{1/2})^{k(S)-1} z^{r(S)}.$$ 
As before, we can write $J_L(t, z)= \sum_{i=0}^g \Phi_i(t) z^{i}.$
\end{definition}

\begin{remark}\label{jones}
The reduced Jones-Krushkal polynomial $J_L(t,z)$ specializes to the usual Jones polynomial $V_L(t)$ under setting $z=-t^{-1/2} -t^{1/2}$. 

In particular, if $L$ is a classical link, then any classical link diagram for $L$ will have $r(S)=0$ for all states. Thus $J_L(t, z)=V_L(t)$ when $L$ is classical.
\end{remark}

For classical links, Jones proved that $V_L(t) \in t^{(m-1)/2} \ZZ[t,t^{-1}],$ where $m$ is the number of components in $L$ \cite[Theorem 2]{Jones-85}. This result was extended to checkerboard colorable virtual links by Kamada, Nakabo, and Satoh \cite[Proposition 8]{KNS-2002}. The next result gives the analogous statement for the reduced Jones-Krushkal polynomial.

\begin{proposition} \label{prop:special-form}
Let $L$ be a checkerboard colorable link in $\Si \times I$ with $m$ components, and let $ J_L(t, z)= \sum_{i=0}^{g} \Phi_{i}(t) z^{i}$. Then $\Phi_{i}(t) \in t^{(m+i+1)/2} \, \ZZ[t,t^{-1}]$. 
\end{proposition}

\begin{proof}
Let $D$ be a checkerboard colorable diagram for $L$. Lemma \ref{lemma:f-poly} implies that $\Phi_{i}(t) \in \ZZ[t^{1/2},t^{-1/2}]$. For any state $S \in \sS(D)$, define $\varphi_S(t) \in \ZZ[t^{1/2},t^{-1/2}]$ by setting
\begin{equation} \label{eqn:onetwo}
\varphi_S(t) z^{r(S)} = (-1)^{-3w(D)}t^{3 w(D)/4}t^{(b(S)-a(S))/4}(-t^{-1/2}-t^{1/2})^{k(S)-1}z^{r(S)}.
\end{equation}

If $S' \in \sS(D)$ is any other state, then we claim that 
\begin{equation} \label{eqn:state}
b(S)-a(S)+2|S| \equiv b(S')-a(S')+2|S'| \; \text{(mod 4)}.
\end{equation}
Since any state can be obtained from any other by switching the smoothings at finitely many crossings, it is sufficient to prove \eqref{eqn:state} when $S'$ is obtained from $S$ by switching just one smoothing. In that case, we have $a(S)-b(S) = a(S')-b(S') \pm 2$ and $|S| = |S'| \pm 1$ by checkerboard colorability, and \eqref{eqn:state} follows.

From \eqref{eqn:onetwo}, it is clear that $\varphi_{S}(t)\in t^{(3w(D)+b(S)-a(S) +2k(S)-2)/4} \, \ZZ[t,t^{-1}]$.

Let $S_\si$ be the Seifert state with all oriented smoothings (see Figure \ref{Fig:oriented}). Recall that $b(S_{\si}) - a(S_{\si})  = -w(D)$. 

We claim that 
\begin{equation} \label{eqn:parity} 
m \equiv |S_{\si}|+n \; \text{(mod 2)}.
 \end{equation} 
Each time we perform an oriented smoothing, the number of components changes by one. Thus the  claim now follows easily by induction on $n$.

For any state $S$, Equation \eqref{eqn:state} implies that 
\begin{equation} \label{eqn:congruence}
b(S)-a(S)+2|S| = b(S_{\si})-a(S_{\si})+2|S_{\si}|+4\ell
\end{equation} for $\ell \in \ZZ$. 
By Equation \eqref{eqn:congruence}  and the fact that $|S| = k(S) +r(S)$, we get that
\begin{eqnarray*}
(3w(D)+b(S)-a(S)+2k(S)-2)/2 &=& (3w(D)+b(S)-a(S)+2|S|-2r(S)-2)/2,\\ 
&=&(2w(D)+2|S_{\si}|+4\ell -2r(S)-2)/2, \\
&=&w(D)+|S_{\si}|+2\ell -r(S)-1,\\
& \equiv& w(D)+m-n-r(S)-1\; \text{(mod 2)},\\
&\equiv& m+ r(S)+1\; \text{(mod 2)}.
\end{eqnarray*}
(The last two steps use Equation \eqref{eqn:parity} and the fact that $n \equiv w(D) \text{(mod 2)}.$) For each state $S$ with $r(S)=i$, we have shown that $\varphi_{S}(t)\in t^{(m+i+1)/2}\ZZ[t,t^{-1}]$.  Since we can write $\Phi_{i}(t) = \sum_{r(S)=i} \varphi_S(t) z^i,$ the proposition now follows.  
\end{proof}

\begin{figure}[!ht]  
\centering
\includegraphics[height=28mm]{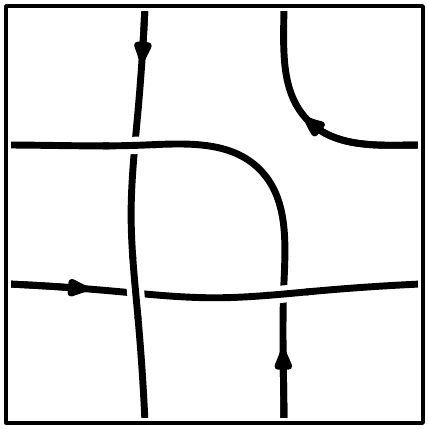}\hspace{.8cm}
\includegraphics[height=28mm]{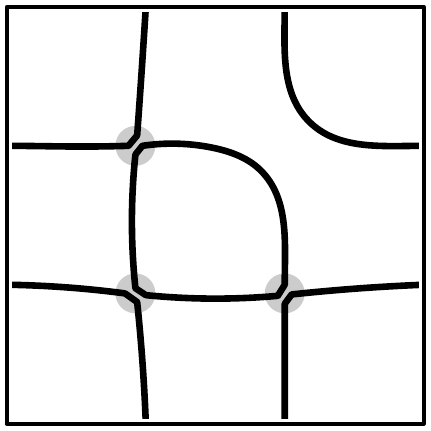}\hspace{.8cm}
\includegraphics[height=28mm]{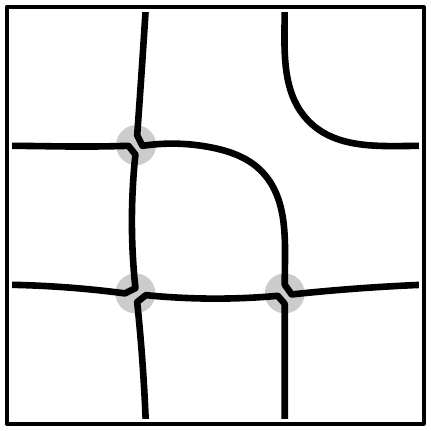} 
\caption{A minimal genus diagram of the virtual Borromean rings in the torus, and the states $S_A$ and $S_B$.} \label{fig:Borromean-in-torus}
\end{figure}

\begin{example}\label{bor}
The virtual Borromean rings $L$ (see Figure \ref{fig:examples}) admits a minimal genus diagram on the torus $T$, which is shown in Figure \ref{fig:Borromean-in-torus} along with $S_A$ and $S_B$. Notice that the diagram is checkerboard colorable, and it has $|S_A| =2, r(S_A) = 0,$ and  $|S_B| =1, r(S_B) = 0.$ By direct computation, the three states $AAB, ABA, BAA$ all have $|S| = 1$ and $r(S)=0,$ and the three states
$ABB, BAB, BBA$ all have $|S| = 2$ and $r(S)=1.$ Therefore, $\lb D \rb_T = A^3 d^2 + 3A d + 3A^{-1} d z + A^{-3}d$, where $d=-A^2-A^{-2}.$ Since $D$ has writhe $w=3$, it follows that
$$J_{L}(t,z) = t - 2 t^2 -t^{3}-3t^{5/2}z.$$
\end{example}

\subsection{Calculations}
In this section, we provide some sample calculations of the homological Kauffman bracket and Jones-Krushkal polynomials. 

\begin{figure}[!ht]
\includegraphics[scale = 0.75]{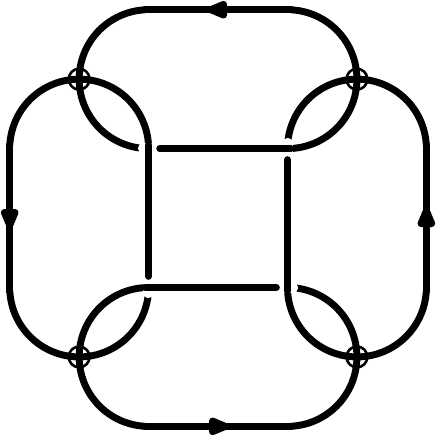}  \qquad
\centering\includegraphics[height=28mm]{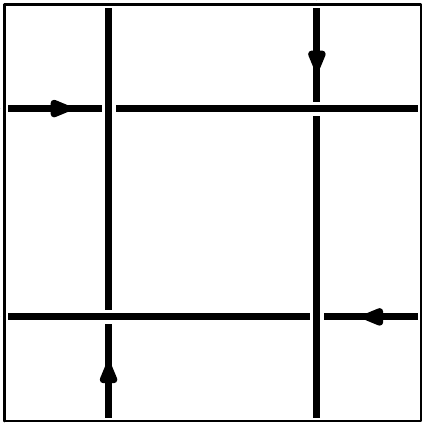}  \quad 
 \includegraphics[height=28mm]{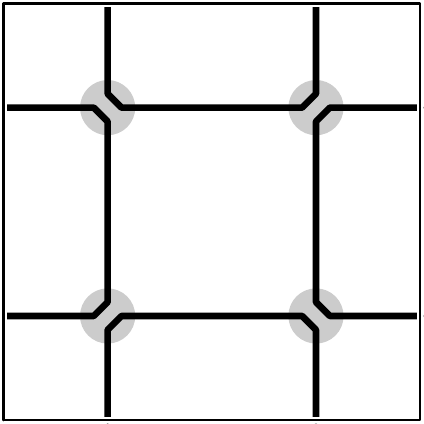}\quad
\includegraphics[height=28mm]{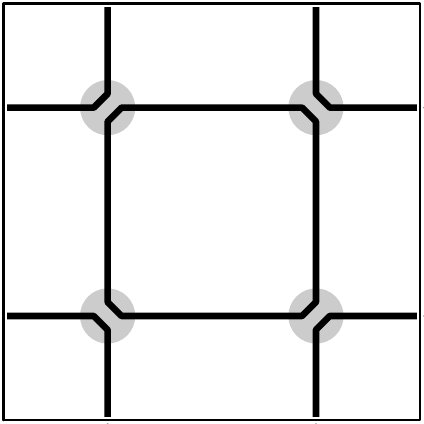} 

\caption{From left to right, a virtual link with four components, a minimal genus representative in the torus, and the states $S_A$ and $S_B$.} \label{fig:four-component-link}
\end{figure}
 
\begin{example} \label{example-four-component-link}
A virtual link $L$ with four components along with a minimal genus diagram on the torus $T$ appear on the left of Figure \ref{fig:four-component-link}. The states $S_A$ and $S_B$ are shown to the right with shading around the smoothed crossings. From this, we see that $|S_A| = 2 = |S_B|$ and $r(S_A) = 0 = r(S_B).$ Resmoothing one of crossings of $S_A$, one can show that the four states $AAAB, AABA, ABAA,$ and $BAAA$ all have $|S|=1$ and $r(S)=0$. Likewise, resmoothing one of crossings of $S_B$, one can similarly show that the four states $ABBB, BABB, BBAB,$ and $BBBA$ all have $|S|=1$ and $r(S)=0$. Resmoothing two of the crossings of $S_A$ (or doing the same to $S_B$), one can show that the six states $AABB, ABBA, BBAA, BAAB, ABAB, BABA$ all have $|S| = 2$ and $r(S) =1.$ Thus,
$$\lb L \rb_T = A^4d^2 + 4A^2d  +6dz+4A^{-2}d +A^{-4}d^2,$$
where $d=-A^2 -A^{-2}.$
 
Since this link has writhe $w(L)=-4$, it follows that
\begin{eqnarray*}
J_{L}(t,z)&=&t^{-3}\left(t^{-1}(-t^{-1/2}-t^{1/2})+4t^{-1/2}+6z+4t^{1/2}
+t(-t^{-1/2}-t^{1/2})\right),\\
&=&\left(-t^{-9/2}+3t^{-7/2}+3t^{-5/2}-t^{-3/2}\right)+6t^{-3}z.
\end{eqnarray*}
The diagram is alternating, therefore it is adequate.
\end{example}

\begin{figure}[!ht]
\centering\includegraphics[scale = 1.25]{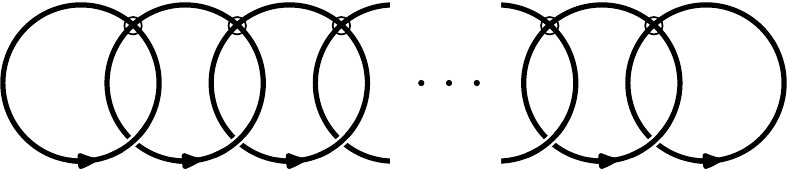} 
\caption{The virtual chain link.} \label{fig:newton}
\end{figure}

\begin{example}\label{newton} 
Figure \ref{fig:newton} shows  virtual chain link with $m-1$ crossings and $m$ components. It is not checkerboard colorable, thus it follows that $r(S) \geq 1$ for all states. One can further show that its virtual genus is $\lfloor \tfrac{m}{2}\rfloor$.

Figure \ref{fig:newton} shows the states $S_A$ and $S_B$ for the virtual chain link. Notice that $|S_A| = 1 = |S_B|$. One can further show that every state $S$ has $|S|=1$. Since $L$ is not checkerboard colorable, it follows that every state has $r(S)=1.$ As a result we have
$$ \lb D \rb_\Si=A^{m-1}z + \binom{m-1}{1}A^{m-3} z +\cdots+\binom{m-1}{m-2}A^{3-m}z+A^{m-1}z=(A+A^{-1})^{m-1}z. $$
Since this link has writhe $w=1-m$, we conclude that
\begin{eqnarray*}
J_{D}(t,z) &=& (-1)^{1-m}t^{(3-3m)/4}\left(t^{-1/4}+t^{1/4}\right)^{m-1}z, \\
&=&(-1)^{m-1}\left(t^{-1}+t^{-1/2}\right)^{m-1}z.
\end{eqnarray*}
\end{example} 

\begin{figure}[!ht]  
\centering
\includegraphics[scale = 0.90]{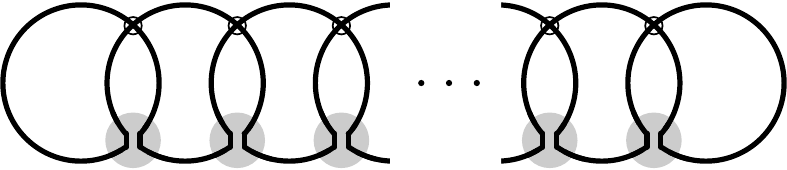}\qquad
\includegraphics[scale = 0.90]{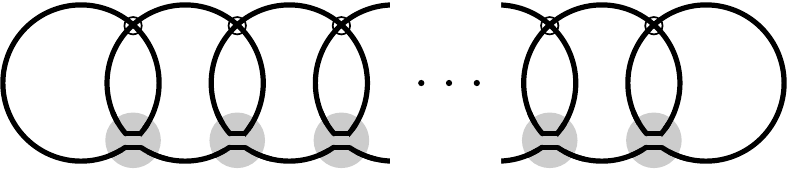} 
\caption{The states $S_A$ and $S_B$ for the virtual chain link.} \label{resolution-newton}
\end{figure}

\subsection{Horizontal and vertical mirror images}
In this section, we describe how the Jones-Krushkal polynomial changes under taking mirror images. Recall that there are two ways to take the mirror image of a virtual link $L$. They are called the vertical and horizontal mirror images, and they are defined in terms of virtual link diagrams as follows.

Given a virtual link diagram $D$, the vertical mirror image is denoted $D^{*}$ and it is the diagram obtained by switching the over and under crossing arcs at each classical crossing, see Figure \ref{fig:3-1}. The horizontal mirror image is denoted $D^{\dag}$ and it is the diagram obtained by reflecting the diagram $D$ across a vertical line $x=x_0$ in $\RR^2$ to the far left of $D$, see Figure \ref{fig:3-1}.  

\begin{figure}[!ht]  
\centering
\includegraphics[height=35mm]{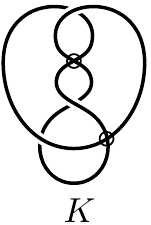}\hspace{.8cm}
\includegraphics[height=35mm]{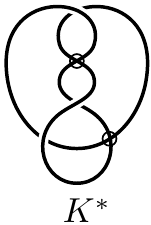}
\hspace{.8cm}
\includegraphics[height=35mm]{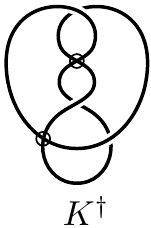}
\caption{The virtual knot $K=3.1$ and its mirror images.} \label{fig:3-1}
\end{figure}

We describe these operations for links in thickened surfaces. Let $L$ be a link in $\Si \times I$, and let $D$ be its diagram on $\Si$. Let $\phi \colon \Si \times I\to \Si \times I$ be the orientation-reversing map given by $\phi(x,t)=(x,1-t)$ for $(x,t) \in \Si \times I$. Then  $\phi(L)=L^{*}$, the vertical mirror image of $L$. Now let $\psi:\Si \to \Si$ be an orientation-reversing homeomorphism. (For example, $\psi$ could be reflection through a plane when $\Si$ is embedded in $\RR^3$.) Then under  $\psi\times\text{id}:\Si\times I \to \Si\times I$, we have  $(\psi\times\text{id})(L)=L^{\dag}$, the horizontal mirror image of $L$.     

\begin{proposition} \label{star}
If $L$ is a link in $\Si \times I$, then 
$$ \wt{J}_{L^*}(t,z)=\wt{J}_{L}(t^{-1},z), \quad \text{ and } \quad \wt{J}_{L^\dag}(t,z)=\wt{J}_{L}(t^{-1},z).$$
If $L$ is a checkerboard colored link in $\Si \times I$, then 
$$J_{L^*}(t,z)= J_{L}(t^{-1},z), \quad \text{ and } \quad J_{L^\dag}(t,z)= J_{L}(t^{-1},z).$$
\end{proposition}  

\begin{proof}
We give the proof for the vertical mirror image; the proof for horizontal mirror image is similar and left to the reader. Let $D$ be a diagram on $\Si$ for $L$. Then $D^*$ is obtained by switching the over and under arcs at each crossing of $D$. Notice that an $A$-smoothing applied to a crossing of $D$ has the same effect as a $B$-smoothing applied to the crossing of $D^*$. Thus, there is a one-to-one correspondence $S \leftrightarrow S^*$ between the state spaces $\sS(D)$ and $\sS(D^{*})$, where $S \in \sS(D)$ and $S^* \in \sS(D^{*})$ have opposite smoothings at the corresponding crossings of $D$ and $D^*$.

Clearly,  $w(D^{*})=-w(D),\ a(S^*)=b(S),$ and $b(S^*)=a(S).$ Further,  we have $|S^*| = |S|, \ k(S^*) = k(S),$ and $r(S^*) = r(S).$ Set $d = -t^{-1/2}-t^{1/2}$. Thus

\begin{eqnarray*} 
\wt{J}_{L^*}(t,z)&=&(-1)^{w(D^*)}t^{3w(D^*)/4} \sum_{S* \in \sS(D^*)} t^{(b(S^*)-a(S^*))/4 }d^{k(S^*)}z^{r(S^*)}, \\
&=&(-1)^{w(D)}t^{-3w(D)/4} \sum_{S \in \sS(D)} t^{(a(S)-b(S))/4} d^{k(S)}z^{r(S)}
 \;\; = \;\; \wt{J}_{L}(t^{-1},z).  
\end{eqnarray*}
The proof for reduced Jones-Krushkal polynomial is similar and is left to the reader.  
\end{proof}

\begin{figure}[!ht]  
\centering
\includegraphics[height=36mm]{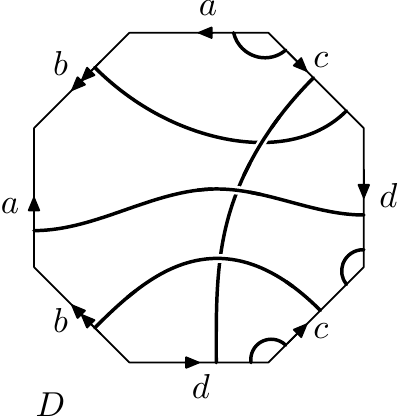}\hspace{.8cm}
\includegraphics[height=36mm]{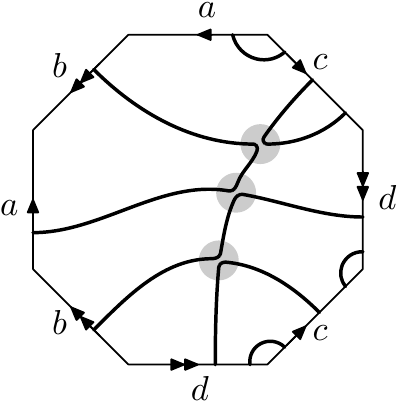}\hspace{0.8cm}
\includegraphics[height=36mm]{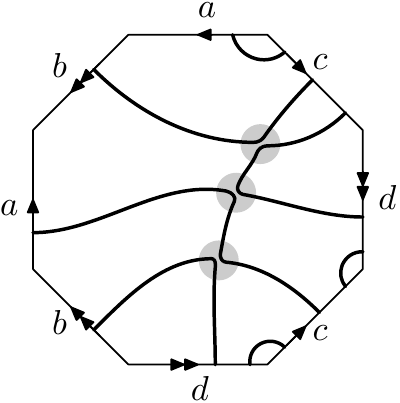} 
\caption{A minimal genus diagram of $3.1$ and the states $S_A$ and $S_B$.} \label{fig:3-1surface}
\end{figure}

\begin{example} \label{example-3-1}
Figure \ref{fig:3-1} shows the virtual knot $K=3.1$ and its mirror images $K^*, K^\dag$,
and  Figure \ref{fig:3-1surface} shows a minimal genus diagram of $K$ on a genus $2$ surface $\Si_2$. The states $S_A$ and $S_B$ are shown in Figure \ref{fig:3-1surface} with shading around the smoothed crossings. From that, one can see that $|S_A| = 1, r(S_A) = 1$ and $|S_B|=1, r(S_B) =1.$ Resmoothing one of the crossings in $S_A$, one can show that two of the three states $AAB, ABA, BAA$ have $|S| =1$ and $r(S)=1$, and the third has $|S|=2$ and $r(S)=2.$ Resmoothing one of the crossings in $S_B$, one can further show that two of the three states $ABB, BAB, BBA$ have $|S| =2$ and $r(S)=2$, and the third has $|S| =1$ and $r(S)=1$.

Since $w(D)=-1,$ we have
\begin{eqnarray*}
\lb D \rb_{\Si_2} &=&A^3z+2Az+Az^2+2A^{-1}z^2 +A^{-1}z+A^{-3}z,\\
\wt{J}_{K}(t,z)&=&-\left(t^{-3/2}+2t^{-1}+t^{-1/2}+1\right)z-\left(t^{-1}+2t^{-1/2}\right)z^2.
\end{eqnarray*}
Thus, Proposition \ref{star} applies to show
$$\wt{J}_{K^*}(t,z)=\wt{J}_{K^\dag}(t,z)=-\left(1+t^{1/2}+2t+t^{3/2}\right)z-\left(2t^{1/2}+t\right)z^2.$$
\end{example}


\section{A Kauffman-Murasugi-Thistlethwaite theorem} \label{sec-4}
In this section, we prove the Kauffman-Murasugi-Thistlethwaite theorem for alternating links in surfaces.

\begin{theorem} \label{thm:KMT}
Let $L$ be a link in the thickened surface $\Si \times I$. If $L$ admits a connected reduced alternating link diagram on $\Si$ with $n$ crossings, then any other link diagram for $L$ has at least $n$ crossings.
\end{theorem}

\begin{proof}
Let $D$ be an arbitrary connected link diagram on $\Si$ with $n$ crossings.  Lemma \ref{lemma:dual-state} and Corollary \ref{cor-adequate} combine to show that
$$\spn(\lb D \rb_\Si) \leq 2n + 2k(S_A) + 2k(S_B) \leq 4n + 4-4g.$$
 
In case $D$ is a connected, reduced, alternating diagram for the link $L \subset \Si \times I$, Theorem \ref{thm:span} implies that $\spn(\lb D \rb_\Si) = 4n-4g+4.$ If $L$ were to admit a diagram $D'$ with fewer crossings,  then the above considerations would imply that $\spn(\lb D' \rb_\Si) < 4n-4g+4,$ which gives a contradiction to the fact that $\spn(\lb D \rb_\Si) = \spn(\lb D' \rb_\Si)$.
\end{proof}

We now explain how to deduce that the writhes of two reduced alternating diagrams for the same link in $\Si \times I$ are equal. 

\begin{definition}
Given a link diagram $D$ in $\Si \times I$, we define its \emph{$r$-parallel}  $D^r$ to be the link diagram in $\Si \times I$ in which each link component of $D$ is replaced by $r$ parallel copies, with each one repeating the same ``over'' and ``under'' behavior of the original component. 
\end{definition}

\begin{lemma}
If $D$ is $A$-adequate, then $D^r$ is also $A$-adequate. If $D$ is $B$-adequate, then $D^r$ is also $B$-adequate. 
\end{lemma}

\begin{proof}
Let $S_{A}(D)$ be the all $A$-smoothing of $D$ and $S_{A}(D^r)$ the all $A$-smoothing of the $r$-parallel $D^r.$ It is straightforward to check that $S_{A}(D^r)$ is the $r$-parallel of $S_{A}(D)$. Therefore, a cycle in $S_{A}(D^r)$ is self-abutting if and only if it is the innermost strand parallel to a self-abutting cycle of $S_{A}(D)$.

Let $S'$ be the state obtained from $S_A(D)$ by switching the smoothing from $A$ to $B$, and let $S''$ be the state obtained from $S_A(D^r)$ by switching the corresponding crossing on the innermost strand. Switching the smoothing at a self-abutting cycle is either a single-cycle smoothing or increases the number of cycles.

Suppose firstly it is a single cycle smoothing. (This corresponds to case (iii) from the proof of Lemma \ref{lemma3.6}.) Then $|S'| = |S_A(D)|,$ and since $D$ is $A$-adequate, we have $k(S') \leq k(S_A(D))$ and $r(S') \geq r(S_A(D))$. Notice that $S''$ has the same homological rank as $S'$, and that $|S''| = |S_A(D^r)|$. Thus $k(S'') \leq k(S_A(D^r))$ as required.

Now suppose that $|S'|=|S_A(D)|+1.$  (This corresponds to case (ii) from the proof of Lemma \ref{lemma3.6}.) Since $D$ is $A$-adequate, the homological rank increases under making the switch from $S_A(D)$ to $S'$. But $S''$ has the same homological rank as $S'$, thus the same is true under making the switch from $S_{A}(D^r)$ to $S''$. In particular, this shows that $k(S'') \leq  S_{A}(D^r)$. 

A similar argument can be used to show the second part, namely that if $D$ is $B$-adequate, then $D^r$ is also $B$-adequate. The details are left to the reader.
\end{proof}

The next result follows by adapting Stong's argument \cite{Stong-1994} (cf. Theorem 5.13 \cite{Lickorish}). The proof is by now standard, but it is included for the reader's convenience.

\begin{theorem} \label{thm:stong}
Let $D$ and $E$ be two link diagrams on $\Si$ that represent isotopic oriented links in $\Si \times I.$ If $D$ is $A$-adequate, then $n_D-w(D)\leq n_E-w(E)$, where $n_D$ and $n_E$ are the number of crossings of the diagrams $D$ and $E$, respectively.
\end{theorem}

\begin{proof}
Let $\{L_i \mid i=1,\ldots, m\}$ be the components of $L$, and let $D_i$ and $E_i$ be the subdiagrams of $D$ and $E$ corresponding to $L_i$. For each $i =1,\ldots, m,$ choose non-negative integers $\mu_i$ and $\nu_i$ such that $w(D_i)+\mu_i=w(E_i)+\nu_i$. Let $D'_i$ be the result of changing $D_i$ by adding $\mu_i$ positive kinks, and let $E'_i$  be the result of adding $\nu_i$ positive kinks to $E_i$. Notice that $D'$ is still $A$-adequate.

The writhes of the individual components satisfy:
$$w(D'_{i})=w(D_i)+\mu_i=w(E_i)+\nu_i=w(E'_{i}),$$
and the contributions from the mixed crossings of $D'$ and $E'$ are both equal to the total linking number  ${\rm link}(L) = \sum_{i\neq j} \lk(L_i,L_j)$, which is an invariant of the oriented link $L$. It follows that $w(D')=w(E')$. 

For any $r$, take $(D')^r$ and $(E')^r$. Then $w((D')^r)=r^2 w(D')$, because in forming the $r$-parallel of a diagram, each crossing is replaced by $r^2$ crossings of the same sign. The diagrams $(D')^r$ and $(E')^r$, are equivalent and have the same writhe, thus their homological Kauffman brackets must be equal. In particular we have $d_{\max}(\lb (D')^r\rb_\Si)=d_{\max}(\lb (E')^r \rb_\Si)$. Lemma \ref{lemma3.6} now implies that
\begin{eqnarray*}
d_{\max}(\lb(D')^r \rb_\Si)&=& \left(n_D+\sum_{i=1}^m \mu_i\right)r^2+2\left(k(S_{A}(D))+\sum_{i=1}^m \mu_i\right)r,\\
d_{\max}(\lb(E')^r \rb_\Si)&\leq& \left(n_E+\sum_{i=1}^m \nu_i\right)r^2+2\left(k(S_A(E))+\sum_{i=1}^m \nu_i \right)r.
\end{eqnarray*}
Since this is true for all $r$, comparing coefficients of the $r^2$ terms,  we find that: 
\begin{equation}\label{eqn:compare}
 n_D+\sum_{i=1}^m \mu_i \leq n_E+\sum_{i=1}^m \nu_i. 
 \end{equation}
Subtracting $\sum_{i=1}^m (\mu_i +w(D_i))= \sum_{i=1}^m (\nu_i + w(E_i))$ from both sides of \eqref{eqn:compare}, we get that
\begin{equation}\label{eqn:new}
n_D-\sum_{i=1}^m w(D_i)\leq n_E-\sum_{i=1}^m w(E_i).
\end{equation}
Subtracting the total linking number ${\rm link} (L)$ from both sides of \eqref{eqn:new} gives the desired inequality.
\end{proof}

\begin{corollary} \label{cor:KMT}
Let $D$ and $E$ be link diagrams on $\Si$ with $n_D$ and $n_E$ crossings, respectively, for the same oriented link $L$ in $\Si \times I$.
\begin{itemize}
\item[(i)] If $D$ is $A$-adequate, then the number of negative crossings of $D$ is less than or equal to the number of negative crossings of $E$.
\item[(ii)] If $D$ is $B$-adequate, then the number of positive crossings of $D$ is less than or equal to the number of positive crossings of $E$.
\item[(iii)] An adequate diagram has the minimal number of crossings.
\item[(iv)] Two adequate diagrams of an oriented link in $\Si \times I$ have the same writhe.
\end{itemize}
\end{corollary}

\begin{proof}
\noindent (i) Let, $n_{+}$ and $n_{-}$ be the number of positive and negative crossings, respectively. We have 
\begin{eqnarray*} n_D-w(D) &\leq& n_E-w(E),\\
n_{+}(D)+n_{-}(D)-(n_{+}(D)-n_{-}(D))&\leq& n_{+}(E)+n_{-}(E)-(n_{+}(E)-n_{-}(E)),\\
n_{-}(D)&\leq& n_{-}(E).\end{eqnarray*}

\noindent (ii) Use the negative kinks in the proof of Theorem \ref{thm:stong}. It follows that

\begin{eqnarray*}
-n_D+\sum_i \mu_i \geq -n_E+\sum_i \nu_i &\Longrightarrow& n_D-\sum_i \mu_i\leq n_E-\sum_i \nu_i,\\ 
n_D+w(D) \leq n_E+w(E)&\Longrightarrow& n_{+}(D)\leq n_{+}(E).
\end{eqnarray*}

\noindent (iii) Follows from (i) and (ii).

\noindent (iv) From (iii), we have $n_D=n_E$. It follows from Theorem \ref{thm:stong} that
$$n_D-w(D)\leq n_E-w(E)\Rightarrow w(E)\leq w(D).$$
From (ii), we have 
$$n_D+w(D)\leq n_E+w(E)\Rightarrow w(D)\leq w(E).$$
Therefore $w(D)=w(E)$.
\end{proof}
 
The corollary above shows that the first and second Tait Conjectures hold for reduced alternating links in surfaces. 
 

\section{The Tait conjectures for virtual links} \label{sec-5}
In this section, we will prove the first and second Tait conjectures for virtual links using  the results from the previous section. Corollary \ref{cor:KMT} gives the desired conclusion for links in a fixed thickened surface, and it remains to extend the statement to stably equivalent links in thickened surfaces. 

This will be achieved in two steps. In the first step, we will show that any reduced alternating diagram $D$ of a virtual link $L$ has minimal genus. Thus, Corollary \ref{cor:KMT} applies to show that $D$ has minimal crossing number among all minimal genus diagrams for $L$. 
In the second, we will show that any non-minimal genus diagram $D'$ for $L$ has crossing number $n(D') \geq n(D).$ This will be proved by relating the spans of $\lb D' \rb_{\Si'}$ and $\lb D \rb_\Si$. 


We first claim that if $L$ is an alternating virtual link, then any alternating virtual link diagram $D$ for $L$ has minimal genus. There are several ways to prove this. One way is to use a recent result of Adams et al.~from \cite{Adams-2019a} to see that any alternating virtual link diagram for $L$ represents a tg-hyperbolic link $\sL \subset \Si \times I$ in a thickened surface. Therefore, by \cite[Theorem 1.2]{Adams-2019b} tg-hyperbolicity implies that this diagram is a minimal genus representative for $L$. Another way is to use the Gordon-Litherland pairing for links in thickened surfaces \cite{Boden-Chrisman-Karimi-2019}. One can compute that any alternating virtual link diagram for $L$ has nullity equal to zero, which implies that the diagram is minimal genus. 

Thus, any reduced alternating diagram $D$ for $L$ is minimal genus, and Corollary \ref{cor:KMT} implies that any other minimal genus diagram $D'$ for $L$ has $n(D') \geq n(D)$. 

To complete the proof, we must rule out the possibility of a minimal crossing diagram which is not minimal genus. The following conjecture takes care of that and would lead to a direct proof of the Tait conjectures for virtual links. 

\begin{conjecture} \label{conj:last}
Given a virtual link $L$, any minimal crossing diagram for it has minimal genus.
\end{conjecture}

Conjecture \ref{conj:last} is known to be true for virtual knots. The proof is due to Manturov and uses homological parity \cite{Manturov-2013}.  As a consequence, we can give a simple proof of Tait's first and second conjectures for virtual knots.

\begin{theorem} \label{thm:Tait-knot}
Suppose $K$ is a virtual knot admitting an adequate diagram $D$ on a minimal genus surface $\Si$ with crossing number $n(D)$ and writhe $w(D)$. Then any other diagram $D'$ for $K$ has crossing number $n(D') \geq n(D)$. If $D_1$ and $D_2$ are two adequate diagrams of minimal genus for $K$, then $n(D_1)=n(D_2)$ and $w(D_1) =w(D_2).$
\end{theorem}

\begin{proof}
If $D'$ is a minimal crossing diagram for $K$, then Conjecture \ref{conj:last} implies $D'$ has minimal genus. Therefore, since $D$ is also a minimal genus diagram, Corollary \ref{cor:KMT} (iii) applies to show that $n(D) \leq n(D')$. If $D_1$ and $D_2$ are two adequate diagrams of minimal genus for $K$, then Corollary \ref{cor:KMT} applies to show that $n(D_1) =n(D_2)$ and $w(D_1) =w(D_2).$
\end{proof}

We will now show how to prove the Tait conjectures for virtual links without assuming Conjecture \ref{conj:last}. This is achieved by developing an alternative approach that involves comparing the spans of the homological brackets of links related by stabilization moves.

To that end, observe firstly that the homological Kauffman bracket is an invariant of unoriented links in thickened surfaces under regular isotopy, and that the Jones-Krushkal polynomial $\wt{J}_L(t,z)$ is an invariant of oriented links under isotopy and diffeomorphism of the thickened surface. As we have seen, however, neither is invariant under stabilization or destabilization.

Suppose then that $L$ is a virtual link and $D$ is representative link diagram on a surface $\Si$. By the Kamada-Kamada construction, we can assume that the inclusion map $D \hookrightarrow \Si$ is a cellular embedding (cf.~Remark \ref{rem:KK}).
 
If $D$ is not a minimal genus diagram for $L$, then it must admit a destabilizing curve $\ga$. Let $D'$ be the link diagram on the destabilized surface $\Si'$ of genus $g-1$ obtained by destabilizing $\Si$ along $\ga.$ Then it follows the homological Kauffman brackets and Jones-Krushkal polynomials of $D$ and $D'$ are related to one another in a more-or-less straightforward way. Namely, the bracket $\lb D' \rb_{\Si'}$  is obtained from $\lb D \rb_\Si$  by replacing $z$ by $-A^2-A^{-2}$ in some of the terms. The Jones-Krushkal polynomials are related in a similar fashion. Specifically, let $U$ be the subspace of $H_1(\Si)$ generated by $[\ga]$ and its Poincar\'e dual and suppose $S \in \sS(D)$ is a state for $D$ such that $i_*(H_1(S)) \cap U \neq 0$. Then under destabilization, if the homological rank of $S$ drops by one, then we substitute one $z$-factor in $\lb D \rb_\Si$ with $-A^2-A^{-2}.$ Otherwise, if   $i_*(H_1(S)) \cap U = 0$, then the homological rank does not change and we do not make the substitution. In either case, we see that $\spn(\lb D'\rb_{\Si'}) \leq \spn(\lb D \rb_{\Si}) + 4$.

Under ideal circumstances, a minimal genus diagram would be obtained from $D$ after one destabilization, but we may need to repeat this process finitely many times in order to obtain a minimal genus diagram. (This step uses Kuperberg's proof of Theorem \ref{thm:kuperberg}, which tells us that any non-minimal genus representative can be repeatedly destabilized to obtain a minimal genus representative.) Therefore, suppose that $\ga_1 \ldots, \ga_\ell$ are destabilizing curves for $D$, and let $D'$ be the link diagram  on the surface $\Si'$ obtained by destabilizing $\Si$ along $\ga_1 \ldots, \ga_\ell$. Notice that $\Si'$ has genus $g' = g - \ell$ and it is by assumption a surface of minimal genus for $L$.
 
Then as explained above, the bracket $\lb D'\rb_{\Si'}$ can be obtained from $\lb D \rb_\Si$ by substituting $z=-A^2-A^{-2}$ for up to $\ell$ of the $z$-factors in the terms $\lb D \, | \, S \rb_\Si$ for any given state $S \in \sS(D).$  The number of $z$-factors requiring substitution in $\lb D \, | \, S \rb_\Si$  is equal to the dimension of $i_*(H_1(S)) \cap U$, where $U \subset H_1(\Si)$ is the symplectic subspace generated by $[\ga_1], \ldots, [\ga_\ell]$. (Note that $\dim U = 2\ell$, since it also contains the Poincar\'e duals of $[\ga_1], \ldots, [\ga_\ell]$.) With each substitution the span of $\lb D \, | \, S \rb_\Si$ increases by four, thus it follows that $\spn(\lb D'\rb_{\Si'}) \leq \spn(\lb D \rb_\Si) + 4 \ell.$

Lemma \ref{lemma:dual-state} and Corollary \ref{cor-adequate}  imply that 
\begin{equation} \label{eqn:one}
\spn(\lb D \rb_\Si) \leq 4(n-g)+4,
\end{equation}
where $n = n(D)$ is the crossing number of $D$ and $g =g(\Si)$ is the genus of $\Si$.

Now suppose that $D''$ is a reduced alternating diagram for $L$. Then $D''$ necessarily has minimal genus, and since $D'$ is also a minimal genus diagram for $L$, Theorem \ref{thm:kuperberg} implies that $D''$ and $D'$ represent equivalent links in $\Si' \times I$. Therefore 
\begin{equation} \label{eqn:two}
\spn(\lb D''\rb_{\Si'}) = \spn(\lb D'\rb_{\Si'}) \leq  \spn(\lb D \rb_\Si) + 4\ell.
\end{equation}

In addition, Theorem \ref{thm:span} implies that 
\begin{equation} \label{eqn:three}
\spn(\lb D''\rb_{\Si'}) = 4(n'' - g') + 4 = 4(n'' - g + \ell)+4,
\end{equation}
where $n'' = n(D'')$ is the crossing number of $D''$.

Therefore, by Equations \eqref{eqn:one}, \eqref{eqn:two}, and \eqref{eqn:three}, we see that 
\begin{equation}
\spn\left(\lb D''\rb_{\Si'}\right) =4(n'' -g+\ell) + 4  \leq \spn(\lb D\rb_{\Si})+4\ell \leq 4(n -g) + 4 \ell + 4.
\end{equation}
Thus, we conclude from this that $n'' \leq n$. 

\begin{theorem} \label{thm:virtual-tait}
Suppose $L$ is a virtual link admitting a reduced alternating diagram $D$ on $\Si$ with crossing number $n(D)$ and writhe $w(D)$. Then any other diagram $D'$ for $L$ has crossing number $n(D') \geq n(D)$. If $D_1$ and $D_2$ are two reduced alternating diagrams for the same virtual link, then $n(D_1)=n(D_2)$ and $w(D_1) =w(D_2).$
\end{theorem}
 
\begin{remark} \label{remark-thistlethwaite}
In the classical setting, Thistlethwaite proved the following stronger result, namely that a classical link $L$ is alternating and prime if and only if $\spn(V_K(t)) = c(L),$ the crossing number of $L$. One can see by example that this result is not true for virtual knots. In particular, the virtual knots 4.98 and 4.107 are both checkerboard colorable, prime and have the same crossing number and reduced Jones-Krushkal polynomial (see Table \ref{table1}). However, 4.107 is alternating and 4.98 is not. 
\end{remark}
 
A natural question is whether the Tait flyping conjecture can also be extended to alternating virtual links. The flype move is shown in Figure \ref{fig-flype}. For tangles that contain only classical crossings, it is immediate that the flype move does not alter the virtual link type. When the tangle contains virtual crossings, the flype move can change the virtual link type. (This was already noted by Zinn-Justin and Zuber in \cite{matrix-integrals}.) 

The analogue of the Tait flyping conjecture is therefore the assertion that any two reduced alternating link diagrams of the same link are related by a sequence of flype moves by classical tangles.

\begin{problem}
Is the Tait flyping conjecture true for alternating virtual links?      
\end{problem}  

\begin{figure}[!ht]
\centering\includegraphics[height=15mm]{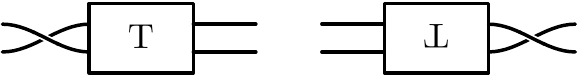}
\caption{The flype move for virtual links, where ``$T$'' is a classical tangle diagram.} \label{fig-flype}
\end{figure}

In a different direction, one can ask whether the Tait conjectures continue to hold in the welded category.

\begin{problem}
Are the Tait conjectures true for alternating welded links?      
\end{problem}  

Another interesting question is whether the Jones-Krushkal polynomial is a virtual unknot detector. Proposition \ref{prop:non-check} shows that a virtual knot that is not checkerboard colorable has nontrivial Jones-Krushkal polynomial.

\begin{problem}
Does there exist a checkerboard colorable virtual knot $K$ which is nontrivial and has $J_K(t,z) =1?$
\end{problem}  

For classical knots, this is equivalent to the open problem which asks whether the Jones polynomial is an unknot detector.


For classical links, Khovanov defined a homology theory that categorifies the Jones polynomial. The result is a  bigraded homology theory of links that is known to detect the classical unknot \cite{KM-2011}. Khovanov homology has been extended to virtual knots and links (see \cites{Manturov-kh, Manturov-2007}), and it categorifies the usual Jones polynomial for virtual links. However, the resulting knot homology does not detect the virtual unknot.
   
An interesting problem would be to construct a triply graded homology theory for links in thickened surfaces that categorifies the Jones-Krushkal polynomial. In particular, is the resulting knot homology theory sufficiently strong to detect the virtual unknot?  

In closing, Table \ref{table1} presents the Jones-Krushkal polynomials for virtual knots with up to three crossings and the reduced Jones-Krushkal polynomial for checkerboard colorable virtual knots with up to four crossings.  

 \subsection*{Acknowledgements}We would like to thank Micah Chrisman, Robin Gaudreau, Andrew Nicas, and Will Rushworth for their valuable feedback.  

\newpage

\renewcommand{\arraystretch}{1.20}
 \begin{table}[H] 
\begin{tabular}{c|c}
Virtual Knot &      
$\wt{J}_K(t,z)$ \\
\hline \hline
2.1 & $(- t^{-5/2}+t^{-3/2}+  t^{-1})z$\\ \hline
3.1 & $-(t^{-3/2}+2t^{-1}+t^{1/2}+1)z -(t^{-1}+2t^{-1/2}) z^2$\\ \hline
3.2 & $ (t^{-2}-t^{-1}+1-t +t^2) z$\\ \hline
3.3 & $-(t^{-3} + 2 t^{-5/2} + 2 t^{-2}) z -(t^{-5/2} +t^{-2} + t^{-3/2}) z^2$\\ \hline
3.4 & $-(3t^{-1} + 2 t^{-1/2}) z -(t^{-3/2} +t^{-1/2} + 1) z^2$\\ \hline
\hline \\
Virtual Knot  &    
$J_K(t,z)$  \\
\hline \hline
3.5 & $(t^{-3} -2t^{-2}) + (t^{-7/2}-t^{-5/2}-t^{-3/2})z$\\ \hline
{\bf 3.6} & $-t^{-4}+t^{-3}+t^{-1}$\\ \hline
3.7 & $(t^{-2}-t^{-1}-1) + (t^{-3/2}-2t^{-1/2})z$\\ \hline
4.85 & $ (3t^{-2}+2t^{-1}) + (t^{-5/2}+6t^{-3/2} +t^{-1/2})z + (t^{-2}+2t^{-1})z^2$  \\ \hline
4.86 & $(-t^{-1}+2-2t) + (-t^{-3/2}+t^{-1/2}-t^{3/2})z$ \\ \hline
4.89 & $ (t^{-4}+4t^{-3}) +(4t^{-7/2}+4 t^{-5/2})z + (2t^{-3}+t^{-2})z^2$ \\ \hline
4.90 & $5 + (4t^{-1/2}+4t^{1/2})z + (t^{-1}+1+t)z^2$ \\ \hline
4.98 & $(t^{-1}+3+t) + (4t^{-1/2}+4 t^{1/2})z + 3z^2$ \\ \hline
4.99  & $(-t^{-1}+3-t) + (-t^{-3/2}+t^{-1/2}+t^{1/2}-t^{3/2})z$ \\ \hline
4.105 & $(t^{-4}+t^{-3}-2t^{-2}+t^{-1}) + (2t^{-7/2}-2t^{-5/2})z$ \\ \hline
4.106 & $(-t^{-3}+t^{-2}-1) +(-t^{-5/2}+2 t^{-3/2}-2t^{-1/2})z$ \\ \hline
4.107 & $(t^{-1}+3+t) +(4t^{-1/2}+4 t^{1/2})z + 3z^2$ \\ \hline
{\bf 4.108}  & $t^{-2}-t^{-1}+1-t+t^2$ \\  \hline
\end{tabular}
\vspace{5mm}
\caption{The Jones-Krushkal polynomials for virtual knots with up to three crossings and for checkerboard colorable virtual knots up to four crossings. Knots listed in boldface font are classical.} \label{table1}
\end{table}


\bibliographystyle{halpha}    

\end{document}